\documentclass[11pt]{article}
\usepackage{amssymb,amsmath,verbatim,amsthm}
\usepackage{booktabs,makecell}
\newtheorem{Theorem}{Theorem}[section]
\newtheorem{Lemma}[Theorem]{Lemma}
\newtheorem{Proposition}[Theorem]{Proposition}

\newtheorem{Example}[Theorem]{Example}

\newtheoremstyle{noparens}
  {}{}
  {\itshape}{}
  {\bfseries}{.}
  { }
  {\thmname{#1}\thmnumber{ #2}\mdseries\thmnote{ #3}}
\theoremstyle{noparens}
\usepackage{enumerate}
\usepackage{color,xcolor}
\usepackage{indentfirst}
\newcommand{\ignore}[1]{}

\begin{document}

\renewcommand{\theequation}{\thesection.\arabic{equation}}

\title{\bf One-factorizations of complete multipartite graphs with distance constraints\footnote{Supported
 by National Natural Science Foundation of China (12571346, 12371326).}}
 \author{
 {\small Yuli  Tan,}  {\small  Junling  Zhou}\footnote{Corresponding Author}\\
 {\small School of Mathematics and Statistics}\\ {\small Beijing Jiaotong University}\\
  {\small Beijing  100044, China}\\
 {\small YL.Tan@bjtu.edu.cn}\\
{\small jlzhou@bjtu.edu.cn}\\
 \and {\small Tuvi Etzion}\\
{\small Department of Computer Science}\\
{\small Technion} \\ %Israel Institute of Technology} \\
 {\small Haifa 3200003, Israel}\\
{\small etzion@cs.technion.ac.il}
}
\date{ }
\maketitle

\begin{abstract}
The present paper considers multipartite graphs from the perspective of design theory and coding theory. A one-factor $F$ of the complete multipartite graph $K_{n\times g}$ (with~$n$ parts of size~$g$) gives rise to a $(g+1)$-ary code ${\cal C}$ of length $n$ and constant weight two. Furthermore, if the one-factor $F$ meets a certain constraint,  then ${\cal C}$ becomes an optimal code with minimum distance three.
 We initiate the study of one-factorizations of complete multipartite graphs subject to distance constraints. The problem of decomposing  $K_{n\times g}$ into the largest subgraphs with minimum distance three is investigated.  It is proved that, for $n\le g$,  the complete multipartite graph $K_{n\times g}$ can be decomposed into $g^2$ copies of the largest subgraphs  with minimum distance three. For even $gn$ with $n>g$, it is proved that the complete multipartite graph $K_{n\times g}$ can be decomposed into $g(n-1)$ one-factors  with minimum distance three, leaving a small gap of $n$ (in terms of $g$) to be resolved
 (If $gn$ is odd when $n>g$, no such decomposition of $K_{n\times g}$ exists).

\medskip

\medskip
\end{abstract}
%%%%%%%%%%%%%%%%%%%%%%%%%%%%%%%%%%%%%%%%%%%%%%%%%%%%%%%%%%%%%%%%%%%%%%%%%

\section{Introduction and preliminaries}

One-factors and one-factorizations of graphs arise naturally in tournament applications and they occur as building blocks in many combinatorial designs and structures.
 The relationships to Steiner triple systems and various Latin squares are canonical topics which stirred up research interest in one-factorizations \cite{new1, new2, new, RS, [141]}. Special types of one-factorizations have been studied extensively, such as
cyclic, perfect, indecomposable, and orthogonal  one-factorizations~\cite{OFbook}. This paper will consider one-factorizations of complete multipartite graphs subject to distance constraints. %where distance is one of the most important ingredients in coding theory.

%Two major application areas are combinatorial arrays and tournaments

%\cite{[85],Anderson ([6],[9], [40]}.

%[6] B. A. Anderson, Finite topologies and Hamiltonian paths. Journal of
%Combinatorial Theory 14B (1973), 87-93.
%[9] B. A. Anderson, Some perfect one-factorizations. Congr:essus Numemntium
%17 (1976), 79-91.

%Constant-weight codes (CWCs) are an important and fascinating class of codes. CWCs with a fixed distance possess error detection and error correction capabilities \cite{code-impo1}. They can improve the reliability of communication on digital channels \cite{apply1} and be used to design oligonucleotide sequences for DNA computing \cite{apply2,apply3}.

We consider decomposition of multipartite graphs from the perspective of constant-weight codes. A one-factor  $F$ of a complete multipartite graph $K_{n\times g}$ (with $n$ parts of size $g$) gives rise
 to a $(g+1)$-ary code ${\cal C}$ of length $n$ and constant  weight two. Furthermore, if the one-factor~$F$ meets a certain constraint, %namely, a distance constraint,
 then  ${\cal C}$ becomes an optimal code with minimum distance three.
 % If the code ${\cal C}$ is optimal (with the largest size) and if it has minimum distance three,  then it corresponds to a one-factor of $K_{n\times g}$ (but with distance constraint).
%For example, an optimal $(g+1)$-ary code of length $n$, constant weight two and minimum distance three with $n> g$ gives rise to a one-factor of a complete multipartite graph $K_{n\times g}$ (with $n$ parts of size $g$).
%On the other hand, a one-factor of a complete multipartite graph $K_{n\times g}$ with some constraint gives rise to a $(g+1)$-ary code of length $n$ and constant weight two.
Similarly we will have that a one-factorization of  $K_{n\times g}$ is a partition of  the set of all  words of length $n$ with weight two over an alphabet of size $q=g+1$ into optimal  $q$-ary codes with minimum distance three. %This motivates to study the problem of  one-factorizations of  complete multipartite graphs with distance constraints. %Let us introduce some formal definitions, common notations, and well-known results in graph theory and coding theory, which will be used in the paper.

A {\em  multipartite graph}, or an  {\em $n$-partite graph}, is  a  graph with vertices partitioned into $n$ parts such that no two vertices within the same part are adjacent. If $n=2$, this defines a {\em bipartite graph.}
A {\em complete}  multipartite graph   is a multipartite graph  such that every pair of vertices in two different parts are adjacent.
The complete multipartite graph with $n$ parts of size $g$ is denoted by $K_{n\times g}$.
By this definition, the complete graph $K_n$ is $K_{n\times 1}$.

A graph $G$ is {\em $r$-regular} if every vertex of $G$ has degree $r$.  It is {\em almost $r$-regular} if
%the degree of each vertex is at most $1$.
every vertex of $G$ has degree $r$ or $r-1$.
A {\em one-factor} is a 1-regular subgraph of $G$. %A one-factor is also called a {\em perfect matching}. A  {\em one-factorization} of $G$ is a set of one-factors of $G$ which are pairwise edge-disjoint  and whose union is all of $G$.In other words, a
A {\em{one-factorization}} of~$G$ is a decomposition of the edge set of~$G$ into edge-disjoint
one-factors. A {\em near one-factor} of~$G$ is a subgraph, which  has one isolated vertex and all other vertices have degree $1$.
A set of near one-factors
which covers every edge of $G$ precisely once is called a {\em near one-factorization} of $G$.
For simplicity, we usually denote a graph solely by its edge set.

%Another approach to the study of one-factors is through matchings.

Let $a,b\in \mathbb{Z}^{+}$ such that $a<b$. The closed integer interval $[a,b]$ is defined as:
\begin{equation*}\label{E00}
[a,b]=\{x\in\mathbb{Z}^{+}:a\leq x \leq b\}.
\end{equation*}
Specifically, the interval $[1,n]$ is abbreviated as $[n]$ for $n\in \mathbb{Z}^{+}$.
%Let $x,y,n\in\mathbb{Z}^{+}$. The notation
%$x+y\in[n]\pmod{n}$ used throughout represents $x+y$ is reduced modulo $n$ to lie in $[n]$.
For a complete graph $K_{2n}$, take the set of vertices to be $[2n-1]\cup\{\infty\}$. For $j\in [2n-1]$, define
$${F}_{j}=\big\{\{\infty,j\},\left\{j+1,j-1\right\},\left\{j+2,j-2\right\},\ldots,\left\{j+n-1,j-n+1\right\}\big\},$$
where $j+i$ and $j-i$ ($i\in[n-1]$) are reduced by modulo $2n-1$ to lie in $[2n-1]$.
It is not difficult to see that  $\big \{{F} _{1},{F} _{2}, \ldots, {F} _{2n-1}\big\}$ forms a one-factorization of $K_{2n}$. A near one-factorization of~$K_{2n-1}$ can be constructed from a one-factorizations of $K_{2n}$ by deleting a single vertex and its incident edges from $K_{2n}$.
For the complete multipartite graph $K_{n\times g}$, it is also well-known that there exists a  one-factorization if and only if $gn$ is even \cite{hamilton}.

For a positive integer $q$, denote  by $\mathbb{Z}_{q}$ the additive group of integers modulo $q$.
Let $\mathbb{Z}_{q}^{n}$ be the set of all words of length $n$ over the alphabet $\mathbb{Z}_{q}$.
A word $\mathbf{x}\in \mathbb{Z}_{q}^{n}$ is denoted by $\mathbf{x} = (x_{1},x_{2},\ldots,x_{n})$. The {\em{(Hamming) distance}} between two words $\mathbf{x}$ and $\mathbf{y}$ is defined as:
$$d(\mathbf{x},\mathbf{y})= |\{i\in [n]:x_{i}\neq y_{i}\}|.$$  The {\em{weight}} of a word $\mathbf{x}$ is the number of nonzero entries in $\mathbf{x}$, i.e., $d(\mathbf{x},\mathbf{0})$.
 The {\em support} of $\mathbf{x}$, denoted by $\text{supp}(\mathbf{x})$, is the set of coordinate positions $i\in[n]$ whose entry $x_{i}$ is nonzero.

Let $\mathcal{H}_{q}(n,w)$ denote the set of all words of length $n$ and weight $w$ over  $\mathbb{Z}_{q}$. A {\em $q$-ary constant-weight code} $\mathcal{C}$ of length $n$, weight $w$ and distance $d$, denoted $(n,d,w)_{q}$-code, is a nonempty subset of $\mathcal{H}_{q}(n,w)$ such that $d(\mathbf{x},\mathbf{y})\geq d$ for all distinct $\mathbf{x},\mathbf{y}\in \mathcal{C}$. Every element of $\mathcal{C}$ is called a {\em{codeword}}.
%The number of codewords in an $(n,d,w)_{q}$-code is called the {\em size} of the code.
The maximum size of an $(n,d,w)_{q}$-code is denoted $A_{q}(n,d,w)$. An {\em optimal}~$(n,d,w)_{q}$-code is an $(n,d,w)_{q}$-code having $A_{q}(n,d,w)$ codewords.
The fundamental problem in coding theory is that of determining $A_{q}(n,d,w)$ and the relevant values in our exposition are as follows.  Throughout the paper, we assume that $q=g+1$.
%However, it is not very hard for weight $w=2$.
\clearpage

\begin{Lemma}[\label{bound_A_q_n2w}\cite{bound_d2}]
\mbox{}
\begin{itemize}
\item[$(1)$] $A_{q}(n,2,w)={{n}\choose{w}}g^{w-1}.$
\item[$(2)$] $A_{q}(n,2w,w)=\lfloor\frac{n}{w}\rfloor.$
\end{itemize}
\end{Lemma}

\begin{Lemma}[\label{n32-bound}\cite{Chee-2007}] %Let $q=g+1$.
%We have that
$$A_{q}(n,3,2)=\min\left\{\left\lfloor\frac{gn}{2}\right\rfloor,{n \choose 2}\right\}=
\begin{cases}
\left\lfloor\frac{gn}{2}\right\rfloor,& \text{ if } n>g,\\
{n \choose 2},& \text{ if } n\le g.\\
\end{cases}
$$
\end{Lemma}

Throughout the paper, let all $n$-partite graphs be defined over the vertex set $V=[n]\times [g]$ with the $n$ parts $P_x=\{x\}\times [g], x\in [n]$, unless otherwise stated.
%If an ordered pair $(x,a)$ represents a point in $K_{n\times g}$, $x$ always reduced modulo $n$ in $[n]$ and $a$ always reduced modulo $g$ in $[g]$.
%Let $G=(V,E)$ be such a multipartite graph.
%with $n$ equi-parts each of size $g$. Denote the vertices by $V=[n]\times [g]$ and the $n$ parts by $P_i=\{i\}\times [g], i\in [g]$.
If arithmetic operations are taken in the expression of a vertex in $V$, then note that the result lies in $V$, usually with the operation for the first component reduced modulo $n$ to lie in $[n]$ and the second reduced modulo $g$ to lie in $[g]$, unless otherwise specified.
Each edge  $e\in K_{n\times g}$, consisting of two vertices~$(x,a),(y,b)$, defines a word $\mathbf{c}_e\in {\cal H}_q(n,2)$, whose $x$-th coordinate is $a$, $y$-th coordinate is~$b$, and all other coordinates are zeros.
As a result, an $n$-partite graph $G$ is equivalent to a constant-weight code $\cal C$ which contains certain 2-weight words.
This establishes a one-to-one correspondence between  multipartite graphs $G$ and  $q$-ary codes  ${\cal C}$ of constant weight two. The  multipartite graph $G$ is said to have {\em distance} $d$ if its associated code  ${\cal C}$ has distance $d$. If the complete multipartite graph $K_{n\times g}$ can be decomposed into the largest subgraphs with a fixed distance $d$, we say that it has an {\em optimal decomposition} with respect to distance $d$. In other words, an optimal decomposition of $K_{n\times g}$ with respect to distance $d$ is a partition of $K_{n\times g}$ into its subgraphs, each of size $A_{q}(n,d,2)$ and with distance $d$.
When $d=2$ or $d=4$, this decomposition is very simple.

\begin{Theorem}\label{d2}
There exists an optimal decomposition of $K_{n\times g}$ with respect to distance $d=2$.
\end{Theorem}

\proof Let ${\cal C}$ be an optimal $q$-ary code of length $n$ and constant weight two. %where $q=g+1$. %We treat distance $d=2$.
If the code  ${\cal C}$ has distance 2,
then by Lemma \ref{bound_A_q_n2w} (1), the size of ${\cal C}$ is ${n\choose 2}g$.
Since the number of edges of $K_{n\times g}$ is ${n\choose 2}g^{2}$, it follows that $K_{n\times g}$ has to be decomposed into $g$ optimal codes.
It is straightforward that the set of supports of the codewords in ${\cal C}$ contains all pairs of $[n]$
and that any two distinct codewords $\mathbf{x},\mathbf{y}$ with the same support $\{i,j\}$ satisfies that $x_i\ne y_i$ and $x_j\ne y_j$.
For any support $\{i,j\}$ and $k\in [g]$, $F_{i,j}^{k} = \{\{(i,m),(j,m+k)\}:1\leq m \leq g\}$ is a one-factor of
the complete bipartite graph with two parts $P_{i}$ and $P_{j}$, where $m+k$ is reduced modulo $g$ to lie in $[g]$.
For any $k\in [g]$, define $H_k= \bigcup_{1\le i<j\le n}F_{i,j}^{k}$. It is readily checked that $\{H_1,H_2,\ldots,H_g\}$ forms an optimal decomposition of $K_{n\times g}$ with respect to distance $2$.
%Equivalently, the associated multipartite graph $G$ of $\mathcal{C}$ contains precisely one perfect matching of every complete bipartite subgraph $G_{i,j}$ over the vertex set $P_{i}\cup P_{j}$ with two different parts $P_i$ and $P_j$, where $1\le i<j\le n$. There are ${n\choose 2}g^{2}$ edges in $K_{n\times g}$, so we need $g$ disjoint such subgraphs of $K_{n\times g}$. It is well-known that there is a one-factorization of each complete bipartite graph $G_{i,j}$, namely,  $\{F_{i,j}^1,F_{i,j}^2,\ldots,F_{i,j}^g\}$ (see \cite[Corollary 4]{OF}). For any $k\in [g]$, define $H_k= \bigcup_{1\le i<j\le n}F_{i,j}^{k}$. It is readily checked that $\{H_1,H_2,\ldots,H_g\}$ forms an optimal decomposition of $K_{n\times g}$ with respect to distance $2$.
\qed

\begin{Theorem}\label{d4}
There exists an optimal decomposition of $K_{n\times g}$ with respect to distance $d=4$.
\end{Theorem}
\proof
Let ${\cal C}$ be an optimal $q$-ary code of length $n$, constant weight two and distance $d=4$ (see Lemma \ref{bound_A_q_n2w} (2)).
It is straightforward that for any two distinct codewords in $\mathcal{C}$, their supports have empty intersection. So the set of supports of $\mathcal{C}$ forms a one-factor of $K_n$ if $n$ is even, a near one-factor of $K_n$ if $n$ is odd.
There are ${n\choose 2}g^{2}$ edges in $K_{n\times g}$, so
we need $(n-1)g^2$ disjoint optimal codes if $n$ is even and $ng^2$ disjoint optimal codes if $n$ is odd. If $n$ is even, then we take a one-factorization of $K_n$ over $[n]$, say $\{F_1,F_2,\ldots,F_{n-1}\}$. For any~$(i,j)\in [g]\times[g]$ and $k\in [n-1]$, define  $G_{(i,j)}^k=\big\{\{(x,i),(y,j)\}:\{x,y\}\in F_k\big\}$. It is immediate that $\big\{G_{(i,j)}^k: (i,j)\in [g]\times [g],\ k\in [n-1]\big\}$ forms an optimal decomposition of $K_{n\times g}$ with respect to distance $4$. If $n$ is odd, the conclusion follows similarly by taking a near one-factorization of~$K_n$ instead. \qed

The rest of the paper is devoted to the existence problem of optimal decompositions of~$K_{n\times g}$ with respect to distance $3$. Section 2 characterizes optimal decompositions of $K_{n\times g}$ in detail and introduces two strategies to deal with general constructions.
Section 3 deals with  the case~$n\le g$ and presents a complete solution for this case.
In Section 4, the case~$n>g$ is considered.
%i.e., the problem of one-factorizations of $K_{n\times g}$ with  distance $d=3$.
We show that generally the complete multipartite graph $K_{n\times g}$ can be decomposed into $g(n-1)$ one-factors  with minimum distance three, leaving a small gap to be resolved. Section 5 is a conclusion with some remarks.

\section{Notations and strategies}
From now on, we concentrate on optimal decompositions of $K_{n\times g}$ with respect to distance~$3$. As a preliminary, this section describes the properties of optimal decompositions, introduces some notations applied throughout this paper, and displays two strategies to construct the optimal codes of the optimal decompositions.

Note that in a binary constant-weight code, the distance between any two codewords is even;
hence, for distance $3$ only $g=q-1>1$ is considered.
%Let ${\cal C}$ be an optimal $q$-ary code of length~$n$ and constant weight two, where $q=g+1$.
%If the code ${\cal C}$ has distance 3, then by Lemma~\ref{n32-bound}, the associated multipartite graph $G$ forms an almost $1$-regular subgraph of $K_{n\times g}$ comprised of exactly $A_{g+1}(n,3,2)=\min\{\left\lfloor\frac{gn}{2}\right\rfloor,{n \choose 2}\}$ edges.
%If $n>g$, then $G$ is a one-factor or a near one-factor of $K_{n\times g}$ according to the parity of $gn$. But the distance of an almost $1$-regular subgraph of $K_{n\times g}$ can be equal to two.
Let $\mathcal{C}$ be an optimal $(n,3,2)_{q}$-code.
It evident from Lemma~\ref{n32-bound} that the associated multipartite graph $G$ of $\mathcal{C}$ forms an almost~$1$-regular subgraph of $K_{n\times g}$ comprised of exactly $A_{q}(n,3,2)=\min\{\left\lfloor\frac{gn}{2}\right\rfloor,{n \choose 2}\}$ edges. However, it is possible that a code obtained from an almost $1$-regular subgraph of $K_{n\times g}$ has distance~$2$.
%Let $G$ be an almost~$1$-regular subgraph of $K_{n\times g}$ comprised of exactly $A_{g+1}(n,3,2)=\min\{\left\lfloor\frac{gn}{2}\right\rfloor,{n \choose 2}\}$ edges.
In fact, the distance between any two disjoint edges whose endpoints lie in the same pair of parts $P_{x}$ and $P_{y}$, $x\neq y$, is $2$.
Therefore, $d=3$ requires that $G$ does not contain `parallel edges',
 where parallel edges represent two vertex-disjoint edges~$e=\{a,b\}$ and~$e'=\{a',b'\}$ with~$a,a'\in P_x$, $b,b'\in P_y$ for some $x,y\in [n]$, $x\ne y$.
For convenience, we use an {\em $\mathrm{AR}$-graph} to represent the largest almost $1$-regular subgraph of $K_{n\times g}$ without parallel edges. (Hence, an AR-graph contains $A_{q}(n,3,2)$ edges). % and comprised of exactly $A_{g+1}(n,3,2)=\min\{\left\lfloor\frac{gn}{2}\right\rfloor,{n \choose 2}\}$ edges. (An $\mathrm{AR}$-graph is also the largest almost $1$-regular subgraph of $K_{n\times g}$ without parallel edges.)
 As a consequence, an optimal decomposition of $K_{n\times g}$ with respect to distance $3$ is a decomposition of $K_{n\times g}$ into edge-disjoint $\mathrm{AR}$-graphs.
 %for which we use the notation $\mathrm{ODAR}(n,g)$.

Note that the number of edges in $K_{n\times g}$ is $g^{2}{n\choose 2}$. If $n\leq g$, an $\mathrm{AR}$-graph has $n\choose 2$ edges and ${n\choose 2}\mid g^{2}{n\choose 2}$; hence, an optimal decomposition of $K_{n\times g}$ with respect to distance $3$ is a decomposition of $K_{n\times g}$ into $g^2$ edge-disjoint $\mathrm{AR}$-graphs, each of size ${n\choose 2}$, for which we use the notation $\mathrm{ODAR}(n,g)$.
If $n>g$ and $gn$ is even, then an $\mathrm{AR}$-graph has $\frac{gn}{2}$ edges and it is a one-factor of $K_{n\times g}$.
Therefore, an optimal decomposition of $K_{n\times g}$ with respect to distance~$3$ becomes a one-factorization of $K_{n\times g}$ with each one-factor not containing parallel edges, for which we denote by $\mathrm{OF}(n,g)$.
If $n>g$ and $gn$ is odd, then an $\mathrm{AR}$-graph is a near one-factor with $\frac{gn-1}{2}$ edges. Since $\frac{gn-1}{2}\nmid g^{2}{n\choose 2}$,~$K_{n\times g}$ cannot be decomposed into near one-factors if $gn$ is odd.
Hence, we always assume that $gn$ is even when~$n>g$.
%For clarity, we state these graph descriptions separately in a lemma.
This is summarized in the following lemma.

\begin{Lemma}\label{d3} %For optimal decompositions of $K_{n\times g}$ with respect to distance $d=3$,
We have the following:
\begin{itemize}
\item If $n\le g$,  an optimal decomposition of $K_{n\times g}$ with respect to distance $3$ is the same as an $\mathrm{ODAR}(n,g)$, a decomposition of $K_{n\times g}$ into $g^2$ copies of $\mathrm{AR}$-graphs, each of size ${n\choose 2}$.
\item If $n>g$ and $gn$ is even, an optimal decomposition of $K_{n\times g}$ with respect to distance $3$ is equivalent to an $\mathrm{OF}(n,g)$, a one-factorization of $K_{n\times g}$
 into $g(n-1)$ one-factors, in which each one-factor does not contain parallel edges and has size $\frac{gn}{2}$.
\item If $n>g$ and $gn$ is odd, there is no optimal decomposition of $K_{n\times g}$ with respect to distance~$3$.
\end{itemize}
\end{Lemma}

 \begin{Example}\label{OF_3_n_2}
{\rm Below we list an $\mathrm{OF}(4,2)$ %an optimal decomposition of $K_{4\times 2}$ with respect to distance $d=3$
in which each row is a one-factor of $K_{4\times 2}$ without parallel edges. An edge $\{(x,a),(y,b)\}$ is simply written as $x_a\ y_b$.}
\end{Example}
%Below we list an $\mathrm{OF}(3,2)$ in which each row corresponds to a one-factor of $K_{3\times 2}$.
%\begin{center}
%\begin{tabular}{l l l l l l l l }
%$1_ 1\text{ }2_ 1$&$1_ 2\text{ }3_ 1$&$2_ 2\text{ }3_ 2$\\
%$1_ 1\text{ }2_ 2$&$2_ 1\text{ }3_ 1$& $1_ 2\text{ }3_ 2$\\
%$1_ 1\text{ }3_ 1$&$2_ 1\text{ }3_ 2$& $1_ 2\text{ }2_ 2$\\
%$1_ 1\text{ }3_ 2$&$1_ 2\text{ }2_ 1$& $2_ 2\text{ }3_ 1$\\
%\end{tabular}
%\end{center}
%The following lists a $\mathrm{TOC}_{3}(4,3,2)$, each row is an optimal $(4,3,2)_{3}$-code with the alphabet $\mathbb{Z}_{3}$.
%The collection $\{{C}_{0},{C}_{1},\ldots,{C}_{5}\}$ is an $\mathrm{OF}(4,2)$, where
\begin{center}
\begin{tabular}{l l l l l l l l }
$1_ 1\text{ }2_ 1$& $1_ 2\text{ }3_ 1$& $2_ 2\text{ }4_ 1$& $3_ 2\text{ }4_ 2$\\
$1_ 1\text{ }2_ 2$& $2_ 1\text{ }3_ 1$& $3_ 2\text{ }4_ 1$& $1_ 2\text{ }4_ 2$\\
$1_ 1\text{ }3_ 1$& $2_ 1\text{ }3_ 2$& $1_ 2\text{ }4_ 1$& $2_ 2\text{ }4_ 2$\\
$1_ 1\text{ }3_ 2$& $2_ 1\text{ }4_ 1$& $3_ 1\text{ }4_ 2$& $1_ 2\text{ }2_ 2$\\
$1_ 1\text{ }4_ 1$& $2_ 1\text{ }4_ 2$& $2_ 2\text{ }3_ 1$& $1_ 2\text{ }3_ 2$\\
$1_ 1\text{ }4_ 2$& $1_ 2\text{ }2_ 1$& $3_ 1\text{ }4_ 1$& $2_ 2\text{ }3_ 2$\\
\end{tabular}
\end{center}
\qed

%For convenience of exposition, we use the following notation.
 For a set ${S}$ of some ordered pairs of $[n]$ and for elements $a,b\in [g]$, denote
\begin{equation}\label{E0}
{S}(a,b)=\big\{\{(x,a),(y,b)\}:(x,y)\in {S}\big\}.
\end{equation}
If ${S}$ is a set of unordered pairs in $[n]$, then we regard every pair as an ordered pair in ascending order and we also adopt the notation (\ref{E0}).
For convenience, relevant notations
and terminology referred to throughout the paper are summarized
in Table 1.

\begin{table}[h!]
\setlength{\abovecaptionskip}{0cm}
\setlength{\belowcaptionskip}{+0.2cm} % µ÷Õû±êÌâÏ·½µÄ¼ä¾à
 \begin{center}
   \caption{Table of definitions and notations}
   \begin{tabular}{c|c|c}
   \toprule[2pt]
     \textbf{Notation} & \textbf{Meaning} & \textbf{Remarks}\\
     \hline
     $[a,b]$ & $\{x\in\mathbb{Z}^{+}:a\leq x \leq b\}$ & Sec.1\\
     $[n]$ & $\{1,2,\ldots,n\}$ & Sec.1\\
     $K_{n\times g}$ & \makecell[c]{Complete multipartite graph with $n$ parts of size $g$, where \\there are no edges between vertices of the same part}& Sec.1\\
     $[n]\times [g]$ & The vertex set of $K_{n\times g}$ &Sec.1\\
     $P_x=\{x\}\times [g]$& A part of $K_{n\times g}, x\in [n]$ &Sec.1\\
     %$(x,a)$& \makecell[c]{a point of $K_{n\times g}$, where $x$ always reduced in $[n]$ and \\$a$ always reduced in $[g]$} &Sec.1\\
      $x_a\ y_b$&
      An edge $\{(x,a),(y,b)\}$& Sec.2\\
     %$K_{n}$ & complete graph & Sec.1\\
     $(n,d,w)_{q}$-code & \makecell[c]{$q$-ary code of length $n$, constant weight $w$, distance $d$}&Sec.1\\
     $A_{q}(n,d,w)$& Maximum size of an $(n,d,w)_{q}$-code &Sec.1\\
     Almost $1$-regular graph&A graph with each vertex having degree at most $1$&Sec.1\\
     $\mathrm{AR}$-graph& \makecell[c]{Almost $1$-regular subgraph of $K_{n\times g}$ without parallel edges\\ and comprised of $A_{q}(n,3,2)$ edges} & Sec.2\\
     $\mathrm{ODAR}(n,g)$&\makecell[c]{Decomposition of $K_{n\times g}$ into  $\mathrm{AR}$-graphs, where $g\geq n$}&Sec.2\\
     $\mathrm{OF}(n,g)$ & \makecell[c]{One-factorization of $K_{n\times g}$, no parallel edges in each \\one-factor, where $n>g$ }&Sec.2\\
     $\delta_{i}$& 0 if $i$ is even;
     1 if $i$ is odd.&Eq.(\ref{E-1})\\
     $S(a,b)$& $\big\{\{(x,a),(y,b)\}:(x,y)\in {S}\big\},S\subset [n]\times[n]$, $(a,b)\in[g]\times[g]$&Eq.(\ref{E0})\\
     $D_{i}$& $\{(x,x+i):x\in [n]\},i\in[n-1]\setminus\{\frac{n}{2}\}$&Eq.(\ref{E1})\\
     ${D}_{\frac{n}{2}}$& $\left\{\left(x,x+\frac{n}{2}\right):x\in\left[\frac{n}{2}\right]\right\},2\mid n$& Eq.(\ref{E2})\\
     ${D}_{n}$ & $\left\{\left(x,x+\frac{n}{2}\right):x\in\left[\frac{n}{2}+1, n\right]\right\},2\mid n$&Eq.(\ref{E3})\\
     $x+y\in[n]\pmod{n}$ & $x+y$ is reduced modulo $n$ to lie in $[n]$& Sec.3, Sec.4\\
     $\mathcal{F}$ & One-factorization or near one-factorization of $K_{n}$ over $[n]$& Sec.3, Sec.4\\
     $\mathcal{G}$ & One-factorization or near one-factorization of $K_{g}$ over $[g]$& Sec.3, Sec.4\\
     \toprule[2pt]
   \end{tabular}
 \end{center}
\end{table}

Employing the graph descriptions in Lemma \ref{d3}, we introduce two general strategies to construct AR-graphs.
Each such graph is an $(n,3,2)_{q}$-code which attains the bound on the largest size. The ideas and variations of the strategies will be used frequently in the remainder of this paper.
%Note that we do not consider the case that $n>g$ and $gn$ is odd in either of the two strategies, since no such optimal decomposition of $K_{n\times g}$ with respect to distance 3 exists.
By Lemma \ref{d3}, if $n>g$, then we only need to consider the case $gn$ is even.

\textbf{Strategy A:}

Let
\begin{equation}\label{s}
s=
\begin{cases}
g,& \text{ if } 2\mid n \text{ and } n>g,\\
n-1,& \text{ if } 2\mid n \text{ and } n\le g,\\
n,& \text{ if } 2\nmid n \text{ and } n\le g.
\end{cases}
\end{equation}
%(We do not consider the case that $2\nmid n$ and $n>g$, for which we may have a slight variation.)
(The case that $2\nmid n$ and $n>g$ is excluded from the present analysis. A slight variation to address this case will be developed later. Note that in this code $g$ must be even.)
Let ${F}_{1},{F}_{2},\ldots,{F}_{s}$ be
 $s$ mutually disjoint one-factors or near one-factors of $K_{n}$ over $[n]$. For any $\{x,y\}\in F_{r}$, $r\in[s]$, we choose two elements $c_{r,x},c_{r,y}\in [g]$ to %form an edge $\{(x,a_{r,x}),(y,a_{r,y})\}$ and $a_{r,x},x\in[n],r\in [s],$
 satisfy the following property:\vspace{-0.2cm}
\begin{itemize}
\item[(A)] For a fixed $z\in[n]$, the set $\{c_{r,z}:r\in [s]\}$ contains $s$ distinct symbols in $[g]$.\vspace{-0.2cm}
\end{itemize}
Now, we construct an $n$-partite graph $${G}=\bigcup_{r\in[s]}{G}_{r},$$ where $${G}_{r}=\big\{\{(x,c_{r,x}),(y,c_{r,y})\}:\{x,y\}\in{F}_{r}\big\}.$$\\[2mm]

%Let ${F}_{1},{F}_{2},\ldots,{F}_{s}$ be $s$ mutually disjoint one-factors or near one-factors of $K_{n}$ over $[n]$. For any $\{x,y\}\in F_{r}$, $r\in[s]$, we choose two elements $a_{r,x},a_{r,y}\in [g]$ to form an edge $\{(x,a_{r,x}),(y,a_{r,y})\}$ and $a_{r,x},x\in[n],r\in [s],$ satisfy the following property:\vspace{-0.2cm}
%\begin{itemize}
%\item[(A)] The set $\{a_{r,x}:r\in [s]\}$, for each fixed $x$, contains $s$ distinct symbols in $[g]$.\vspace{-0.2cm}
%\end{itemize}
%By this, we construct $s$ subgraphs of $K_{n\times g}$ with edge set $${G}_{r}=\big\{\{(x,a_{r,x}),(y,a_{r,y})\}:\{x,y\}\in{F}_{r}\big\}.$$

 %Define an $n$-partite graph with edge set ${A}=\bigcup_{r\in[s]}{A}_{r}$, where $${A}_{r}=\big\{\{(x,a_{r,x}),(y,a_{r,y})\}:\{x,y\}\in{F}_{r},x<y\big\},a_{r,x},a_{r,y}\in [g].$$
  %It is clear that ${G}$ has $A_{g+1}(n,3,2)$ edges (see Lemma \ref{n32-bound}).
%Furthermore, ${G}$ is almost $1$-regular if it possesses the following property:\vspace{-0.2cm}

%Let $(x,a)$ be a point in $K_{n\times g}$.

%(We do not consider the  case that $2\nmid n$ and $n>g$, for which we may have a slight variation.)

\begin{Theorem}\label{SA}
The graph $G$ constructed by Strategy A forms an AR-graph.
\end{Theorem}
\proof Since ${F}_{1},{F}_{2},\ldots,{F}_{s}$ are mutually disjoint one-factors or near one-factors of $K_{n}$, %the endpoints of any two edges in ${G}=\cup_{r\in[s]}{G}_{r}$ belong to at least three different parts. Hence
the~$n$-partite graph ${G}=\cup_{r\in[s]}{G}_{r}$ does not contain parallel edges.
${G}$ has $\frac{sn}{2}$ edges if $n$ is even, and~$\frac{s(n-1)}{2}$ edges if $n$ is odd, i.e., ${G}$ has $A_{q}(n,3,2)$ edges (see Lemma \ref{n32-bound}).
In each $F_{r}$, $r\in[s]$,
for a fixed $z\in [n]$, there is at most one pair containing $z$;
hence, there exists at most one edge incident to
 $(z,c_{r,z})$.
Since the set $\{c_{r,z}:r\in [s]\}$, for each fixed $z$, contains $s$ distinct symbols in~$[g]$,
it follows that there is at most one edge in $G$ incident to $(z,a)$ for any $a\in[g]$.
This means that ${G}$ is almost $1$-regular.
Hence, $G$ is an AR-graph.
\qed

%In fact, if there exist two elements $a_{r_{1},x},a_{r_{2},x}$ such that $a_{r_{1},x}=a_{r_{2},x}=a$, we have two edges $e_{1}\in G_{r_{1}}$ and $e_{2}\in G_{r_{2}}$ that are both incident to vertex $(x,a)$. This contradicts that ${G}$ is almost $1$-regular, as vertex $(x,a)$ has degree greater than one.

%Obviously,  $\mathcal{A}$ has $\sum_{r=1}^{s}|\mathcal{F}_{r}|$ edges.
%When $n$ is even, if $s=g\leq n-1$, then the size of $\mathcal{A}$ is $\frac{gn}{2}$, and thus $\mathcal{A}$ is a one-factor of $K_{n\times g}$ without parallel edges.
%If $s=n-1< g$, then $\mathcal{A}$ is an almost $1$-regular subgraph of $K_{n\times g}$ without parallel edges, and the size of $\mathcal{A}$ is $\frac{n(n-1)}{2}={n\choose 2}$.
%When $n$ is odd, if $s=n\leq g$, $\mathcal{A}$ is an almost $1$-regular subgraph of $K_{n\times g}$ without parallel edges, and the size of $\mathcal{A}$ is $\frac{(n-1)n}{2}={n\choose 2}$.

%\noindent{\textbf{Remark:}}
%If ${F}_{1},{F}_{2},\ldots,{F}_{s}$ are $s$ mutually disjoint near one-factors of $K_{n}$, there exists an $r$ for which $a_{r,x}$ does not exist, and set $\{a_{r,x}:r\in [s]\}$ is defined only under the conditions for which $a_{r,x}$ exists.

\begin{Example}\label{SA}
{\rm  Let $n=g=4$ and $s=n-1=3$.
Take $F_{1}=\{\{1 ,2\},\{3,4\}\}$, $F_{2}=\{\{1,3\},\{2 ,4\}\}$ and $F_{3}=\{\{1,4\},\{2 ,3\}\}$ as a one-factorization of $K_{4}$.
For any $x\in[4]$, define $c_{r,x}=r$, $r\in[3]$ and hence $\{c_{r,x}:r\in [3]\}=[3]$. Using Strategy A, we obtain the AR-graph
\begin{center}
\begin{tabular}{l l l l l l l l }
%$1_ 2\text{ }2_ 2$& $3_ 2\text{ }4_ 2$&
%$1_ 3\text{ }3_ 3$& $2_ 3\text{ }4_ 3$&
%$1_ 4\text{ }4_ 4$& $2_ 4\text{ }3_ 4$\\
$G=\{1_ 1\text{ }2_ 1,\text{ }3_ 1\text{ }4_ 1,\text{ }1_ 2\text{ }3_ 2,\text{ }2_ 2\text{ }4_ 2,\text{ }1_ 3\text{ }4_ 3,\text{ }2_ 3\text{ }3_ 3\}.$\\
%$G_{1}=\{1_ 1\text{ }2_ 1,3_ 1\text{ }4_ 1\}$,\\
%$G_{2}=\{1_ 2\text{ }3_ 2,2_ 2\text{ }4_ 2\}$,\\
%$G_{3}=\{1_ 3\text{ }4_ 3,2_ 3\text{ }3_ 3\}$.\\
%$G_{1}=$&$1_ 1\text{ }2_ 1$& $3_ 1\text{ }4_ 1$\\
%$G_{2}=$&$1_ 2\text{ }3_ 2$& $2_ 2\text{ }4_ 2$\\
%$G_{3}=$&$1_ 3\text{ }4_ 3$& $2_ 3\text{ }3_ 3$\\
\end{tabular}
\end{center}}
\end{Example}
\qed

%\begin{center}
%\begin{tabular}{l l l l l l l l }
%$1_ 2\text{ }2_ 2$& $3_ 2\text{ }4_ 2$&
%$1_ 3\text{ }3_ 3$& $2_ 3\text{ }4_ 3$&
%$1_ 4\text{ }4_ 4$& $2_ 4\text{ }3_ 4$\\
%$1_ 3\text{ }2_ 3$& $3_ 3\text{ }4_ 3$&
%$1_ 4\text{ }3_ 4$& $2_ 4\text{ }4_ 4$&
%$1_ 1\text{ }4_ 1$& $2_1\text{ }3_ 1$\\
%$1_ 4\text{ }2_ 4$& $3_ 4\text{ }4_ 4$&
%$1_ 1\text{ }3_ 1$& $2_ 1\text{ }4_ 1$&
%$1_ 2\text{ }4_2$& $2_ 2\text{ }3_ 2$\\
%$1_ 1\text{ }2_ 1$& $3_ 1\text{ }4_ 1$&
%$1_ 2\text{ }3_ 2$& $2_ 2\text{ }4_ 2$&
%$1_ 3\text{ }4_ 3$& $2_ 3\text{ }3_ 3$\\
%\end{tabular}
%\end{center}}

\textbf{Strategy B:}

Partition all ordered pairs of $[n]$ into $n$ sets ${D}_{1},{D}_{2},\ldots,{D}_{n}$ as follows:
\begin{equation}\label{E1}
{D}_{i}=\left\{\left(x,x+i\right):x\in \left[n\right]\right\}, i\in[n-1] \text{ and } i\ne\frac{n}{2}\text{ if } 2\mid n,
\end{equation}
\begin{equation}\label{E2}
{D}_{\frac{n}{2}}=\left\{\left(x,x+\frac{n}{2}\right):x\in\left[\frac{n}{2}\right]\right\}, \text{ if } 2\mid n,
\end{equation}
\begin{equation}\label{E3}
{D}_{n}=\left\{\left(x,x+\frac{n}{2}\right):x\in\left[\frac{n}{2}+1, n\right]\right\},\text{ if } 2\mid n,
\end{equation}
where $x+i$ is reduced modulo $n$ to lie in $[n]$.
Note that
$(x,y)\in {D}_{i}$ implies $(y,x)\in {D}_{j}$, where $i+j=n$, except $i=\frac{n}{2}$, where $(x,y)\in {D}_{\frac{n}{2}}$ implies $(y,x)\in {D}_{n}$.
 We say that a pair $\{i,j\}\subset[n]$ is {\em forbidden} if $(x,y)\in {D}_{i}$ implies $(y,x)\in {D}_{j}$. Accordingly, we have $\frac{n-1}{2}$ forbidden pairs for odd~$n$
  and $\frac{n}{2}$ forbidden pairs for even $n$. %Furthermore, On $D_{i}$, for any two elements $a,b\in [g]$, we have $D_{i}(a,b)$ is a subgraph of $K_{n\times g}$.

Let
\begin{equation}\label{t}
t=
\begin{cases}
{n\over 2},& \text{ if } 2\mid n \text{ and } n\le g,\\
{n-1\over 2},& \text{ if } 2\nmid n \text{ and } n\le g,\\
{g\over 2},& \text{ if } 2\nmid n,\ 2\mid g, \text{ and } n>g.
\end{cases}
\end{equation}
(The case that $2\mid n$, $2\nmid g$ and $n>g$ is not considered here.)
%Since there exist $\frac{n-1}{2}$ forbidden pairs for odd $n$ and $\frac{n}{2}$ forbidden pairs for even $n$,
It is evident that $t$ is not greater than the number of forbidden pairs.
%(The case that $2\mid n$, $2\nmid g$ and $n>g$ is excluded from the present analysis.)
%We choose $t$ distinct sets $D_{i_{1}},D_{i_{2}}, \ldots,D_{i_{t}}$ satisfying the following property:\vspace{-0.2cm}
Let $D_{i_{1}},D_{i_{2}}, \ldots,D_{i_{t}}$ be $t$ distinct sets satisfying the following property:\vspace{-0.2cm}
\begin{itemize}
%\item[(B1)] the set $\{i_{1},i_{2}, \ldots, i_{t}\}$ contains $t$ distinct elements in $[n]$;
\item[(B)]  $|P\cap\{i_{1},i_{2}, \ldots, i_{t}\}|\leq 1$ for any forbidden pair $P$. %; the  equality holds if $n\le g$.
\end{itemize}
%Let $\{a_{r},b_{r}:r\in[t]\}$ contain $2t$ distinct elements from $[g]$.
%Take a one-factor or near one-factor $G$ of $K_{g}$ over $[g]$,  %say $\left\{\{a_{i},b_{i}\}:k\in\left[\frac{g-\delta_{i}}{2}\right]\right\}$,
%where
%\begin{equation}\label{E-1}
%\delta_{i}=
%\begin{cases}
%0,& \text{ if } 2\mid i,\\
%1,& \text{ if } 2\nmid i.\\
%\end{cases}
%\end{equation}
Take out a subset $\{\{a_r,b_r\}:r\in[t]\}$ from any one-factor or near one-factor of $K_{g}$ over $[g]$, where $a_{r}<b_{r}$.
%Choose $t$ distinct edges $\{a_r,b_{r}\}$, $r\in[t]$ from $G$ and let $a_r<b_r$.
Then
construct an $n$-partite graph $${G}=\bigcup_{r\in[t]}{D}_{i_{r}}(a_{r},b_{r}).$$
%If $n$ is odd, the elements of $D_{i_{r}}$ is $n$. If $n$ is even, the elements of $D_{i_{r}}$ is $n$ if $i_{r} \in [n-1]\setminus \{\frac{n}{2}\}$, $\frac{n}{2}$ if $i_{r}=\frac{n}{2},n$.

%Recall the notation $\mathcal{D}_{i}$ from (\ref{E1}) to (\ref{E3}).
%Now we use the notations  in  (\ref{E0}) and  (\ref{E1})--(\ref{E3}).
%It is not difficult to check that ${G}$ has $A_{g+1}(n,3,2)$ edges.
 %(Here we only consider the three cases involved in (\ref{t}).)
 %For $2\nmid n$ and $n>g$, $t$ is not defined.
 %In order to ensure that ${G}$ does not contain parallel edges, by the fact on $D_i$'s mentioned above, ${G}$ should have the following

%For $2\nmid n$ and $n>g$, $t$ is not defined.
\begin{Theorem}\label{SB}
The graph $G$ constructed by Strategy B forms an AR-graph.
%Strategy A gives one of the largest AR-graphs.
\end{Theorem}
\proof
Since $D_{i_{1}},D_{i_{2}}, \ldots,D_{i_{t}}$ are $t$ distinct sets satisfying the property (B), any ordered pair $(x,y)\in[n]\times[n]$ with $x\neq y$ occurs at most once in $\cup_{r\in [t]}D_{i_{r}}$; and $(x,y)\in \cup_{r\in [t]}D_{i_{r}}$ implies $(y,x)\notin \cup_{r\in [t]}D_{i_{r}}$.
It follows that the $n$-partite graph ${G}=\bigcup_{r\in[t]}{D}_{i_{r}}(a_{r},b_{r})$ does not contain parallel edges.
For any coordinate position $x\in[n]$, there exists at most one ordered pair in $D_{i_{r}}$ such that its first component is $x$; so is for its second component.
Since $\{\{a_{r},b_{r}\}:r\in[t]\}$ is a subset of a one-factor or a near one-factor, it follows that $\{a_{r},b_{r}:r\in[t]\}$ contains~$2t$ distinct elements in~$[g]$.
This means that $G$ is almost $1$-regular.
If $n$ is odd, then $|D_{i}|=n$ for any $i\in[n-1]$ and the number of edges in $G$ equals $\sum_{r\in[t]}|D_{i_{r}}|=tn=A_{q}(n,3,2)$. If $n$ is even and $n\leq g$, then $|D_{i}|=n$ if $i\in[n-1]\setminus\{\frac{n}{2}\}$ and $|D_{i}|=\frac{n}{2}$ if $i\in\{\frac{n}{2},n\}$, which is a forbidden pair. By property~(B), the number of edges in $G$ equals $\sum_{r\in[t]}|D_{i_{r}}|=(\frac{n}{2}-1)n+\frac{n}{2}= A_{q}(n,3,2)$.
It follows that $G$ is an AR-graph.\qed

%$D_{i}$ has $n$ ordered pairs when $i\in[n-1]\setminus\{\frac{n}{2}\}$ if $n$ is even, $D_{i}$ has $\frac{n}{2}$ ordered pairs when $i\in\{\frac{n}{2},n\}$. Since there exist $\frac{n-1}{2}$ forbidden pairs for odd $n$ and $\frac{n}{2}$ forbidden pairs for even $n$, it is evident that $t$ is not greater than the number of forbidden pairs and $G$ has $\Sigma_{r\in [t]}|D_{i_{r}}|=A_{g+1}(n,3,2)$ edges (see Lemma \ref{n32-bound}).

%If $2\mid n $ and $n\le g$, we choose $t=\frac{n}{2}$. In order to satisfy property (B), the set $\{i_{1},i_{2},\ldots,i_{t}\}$ is either $[\frac{n}{2}]$ or $[\frac{n}{2}+1,n]$. If $2\nmid n $ and $n\le g$, we choose $t=\frac{n-1}{2}$. In order to satisfy property (B), the set $\{i_{1},i_{2},\ldots,i_{t}\}$ is either $[\frac{n-1}{2}]$ or $[\frac{n+1}{2},n-1]$. If $2\nmid n $, $2\mid g$ and $n> g$, we choose $t=\frac{g}{2}$. The size of each $D_{i_{r}}$ is $n$ since $2\nmid n $.
%Hence, the number of edges of $G$ is $\Sigma_{r\in [t]}|D_{i_{r}}|=A_{g+1}(n,3,2)$ (see Lemma \ref{n32-bound}).
%For any edge $\{(x,a),(y,b)\}\in G$, we have $(x,y)\in \bigcup_{r\in [t]}D_{i_{r}}$ or $(y,x)\in \bigcup_{r\in [t]}D_{i_{r}}$. property (B) ensures that %the elements of the set $\bigcup_{r\in [t]}D_{i_{r}}$ are represented as unordered pairs consisting of distinct elements. Therefore,
%the graph $G$ does not contain parallel edges. %${G}$ is almost $1$-regular with $A_{g+1}(n,3,2)$ edges and does not contain parallel edges, then

\begin{Example}\label{SB}
{\rm Let $n=g=4$ and $t=\frac{n}{2}=2$.
By (\ref{E1})-(\ref{E3}),
$D_{1}=\{(1,2),(2,3),(3,4),(4,1)\}$, $D_{2}=\{(1,3),(2,4)\}$, $D_{3}=\{(1,4),(2,1),(3,2),(4,3)\}$, and $D_{4}=\{(3,1),(4,2)\}$.
Let $(a_1,b_1)=(1,2)$, $(a_{2},b_2)=(3,4)$.
Using Strategy B, we obtain the AR-graph
$$D_{1}(1,2)\cup D_{2}(3,4) = \{1_ 1\text{ }2_ 2, \text{ }2_ 1\text{ }3_ 2, \text{ }3_ 1\text{ }4_ 2, \text{ }4_ 1\text{ }1_ 2,\text{ }1_ 3\text{ }3_ 4, \text{ }2_ 3\text{ }4_ 4\}.$$}
\qed
\end{Example}

\section{Solution to $\mathrm{ODAR}(n,g)$s}

In this section, the complete solution to existence of $\mathrm{ODAR}(n,g)$s will be given. %for which  we will provide a complete solution.
In order to construct a series of AR-graphs over the vertex set $[n]\times [g]$, we take arithmetic operations frequently. For integers $x,y$, we use $x+y\in[n]\pmod{n}$ to denote that the sum~$x+y$ is reduced modulo $n$ to lie in $[n]$. To alleviate the notation, we even write this as $x+y$ if the context is clear. Similarly, for integers $i,j$, we use the notation $i+j\in[g]\pmod{g}$ or simply~$i+j$ if no confusion arises. %Prior to presenting the next theorem, we introduce a notation.
In the sequel let
\begin{equation}\label{E-1}
\delta_{i}=
\begin{cases}
0,& \text{ if } 2\mid i,\\
1,& \text{ if } 2\nmid i.\\
\end{cases}
\end{equation}

\begin{Theorem}\label{ODAR_even_n}
There exists an $\mathrm{ODAR}(n,g)$ for any positive integer $g$ and even $n$ with $n\le g$.
\end{Theorem}
\proof Let $\mathcal{G}=\{{G}_{1},{G}_{2},\ldots,{G}_{g-1+\delta_g}\}$ be a one-factorization if $g$ is even or a near one-factorization if $g$ is odd, of the complete graph $K_{g}$ over $[g]$.
%List the pairs in ${G}_i$ by  %$${G}_{i}=\left\{\{a_{i,1},b_{i,1}\},\{a_{i,2},b_{i,2}\},\ldots,\big\{a_{i,\frac{g-\delta_g}{2}},b_{i,\frac{g-\delta_g}{2}}\big\}\right\},\ \text{where} \ a_{i,w}<b_{i,w}, w\in[\frac{g-\delta_g}{2}].$$
Assign labels to each element in ${G}_i$ by
\begin{align}\label{GGG}{G}_{i}=\left\{\{a_{i,w},b_{i,w}\}:w\in\left[\frac{g-\delta_g}{2}\right],a_{i,w}<b_{i,w} \right\}.\end{align}
%Assume w.l.o.g. that $a_{i,w}<b_{i,w}$ for $i\in[g-1+\delta_g]$ and $w\in[\frac{g-\delta_g}{2}]$.
%We distinguish between two cases depending on the parity of $n$. %we discuss two cases.

%\textbf{Case 1:} $n$ is even.

Let $\mathcal{F} = \{{F}_{1},{F}_{2},\ldots,{F}_{n-1}\}$ be a
one-factorization of the complete graph $K_{n}$ over $[n]$.
Since~$n$ is even and $n\le g$, we apply Strategy A and take $s=n-1$ by (\ref{s}).
Fix $i\in[g]$.
For any $\{x,y\}\in F_{r}$, $r\in[n-1]$, let $c_{r,x}=c_{r,y}=i+r\in [g]\pmod{g}$.
Since $n\le g$, for any fixed $z\in[n]$, the set $\{c_{r,z}:r\in[n-1]\}$ contains~$n-1$ distinct elements in $[g]$.
%where $i+r\in[g]\pmod{g}$. %is reduced modulo $g$ to lie in $[g]$.
So property (A) is satisfied.
Therefore, we get an AR-graph
$${A}_{i} = \bigcup_{r\in[n-1]}{F}_{r}(i+r,i+r),$$
where $i+r\in[g]\pmod{g}$. %$i+r$ modulo $g$ is reduced to lie in $[g]$.
%Obviously, these $g$ AR-graphs $A_{i}$, $i\in[g]$, are edge-disjoint graphs.
When $i$ runs over $[g]$, we obtain $g$ AR-graphs $A_{i}$, $i\in[g]$. Now, we run out of all edges of the form $\{(x,a),(y,a)\}$, $x,y\in[n]$, $x\neq y$, $a\in[g]$ as $\mathcal{F}$ is a one-factorization of $K_{n}$.

Since an $\mathrm{ODAR}(n,g)$ consists of $g^{2}$ subgraphs, we still need $g(g-1)$ pairwise disjoint AR-graphs by Lemma \ref{d3}.
Since $n$ is even and $n\le g$, apply Strategy B with $t={n\over 2}$ by (\ref{t}).
Clearly, $D_{1},D_{2}, \ldots,D_{{n\over 2}}$ are ${n\over 2}$ distinct sets satisfying property (B).
For a fixed $i\in [g-1+\delta_g]$, since ${G}_{i}$ is a one-factor or near one-factor of $K_{g}$, for any $j\in[\frac{g-\delta_{g}}{2}]$, we construct an AR-graph
$${B}_{i,j}^{1}=\bigcup_{r \in [\frac{n}{2}]}{D}_{r}(a_{i,j+r},b_{i,j+r}),$$
where $j+r\in[\frac{g-\delta_{g}}{2}]\pmod{\frac{g-\delta_{g}}{2}}$. %is reduced modulo $\frac{g-\delta_{g}}{2}$ to lie in  $[\frac{g-\delta_{g}}{2}]$.
Similarly, $D_{{{n}\over 2}+1},D_{{{n}\over 2}+2}, \ldots,D_{n}$ are ${n\over 2}$ distinct sets satisfying property (B).
Then, for any $i\in [g-1+\delta_g]$ and $j\in[\frac{g-\delta_{g}}{2}]$, we also construct an AR-graph
$${B}_{i,j}^{2}=\bigcup_{r\in[\frac{n}{2}+1,n]}{D}_{r}(a_{i,j+r},b_{i,j+r}),$$
where $j+r\in[\frac{g-\delta_{g}}{2}]\pmod{\frac{g-\delta_{g}}{2}}$. %is reduced modulo $\frac{g-\delta_{g}}{2}$ to lie in  $[\frac{g-\delta_{g}}{2}]$.
%Because $n\le g$, we have $\frac{g-\delta_{g}}{2}\geq\frac{n}{2}$. Thus the sets $\{a_{i,j+r},b_{i,j+r}:r\in [\frac{n}{2}]\}$ and $\{a_{i,j+r},b_{i,j+r}:r\in[\frac{n}{2}+1, n]\}$ contain $n$ distinct elements in $[g]$, respectively. %Now that Property (B) is satisfied,
%Hence, both ${B}_{i,j}^{1}$ and ${B}_{i,j}^{2}$ form  AR-graphs. %the largest almost $1$-regular subgraph of $K_{n\times g}$ without parallel edges.
Thus we obtain $(g-1+\delta_{g})\cdot \frac{g-\delta_{g}}{2}\cdot 2=g(g-1)$ AR-graphs~$B_{i,j}^{k}$,~$ i\in [g-1+\delta_{g}]$, $j\in [\frac{g-\delta_{g}}{2}]$, $k\in[2]$.

%Note that an AR-graph contains exactly $n\choose 2$ edges, $K_{n\times g}$ contains exactly $g^{2}{n\choose 2}$ edges and we construct $g^{2}$ AR-graphs.
Note that we have obtained $g^{2}$ AR-graphs, by Lemma \ref{d3}, we only need to prove that every edge $\{(x,a),(y,b)\}$ in $K_{n\times g}$ occurs at least once in the collection $\{{A}_{i}:i\in[g]\}\bigcup \{{B}_{i,j}^{k}:i \in [g-1+\delta_{g}], j\in [\frac{g-\delta_{g}}{2}], k\in[2]\}$. Thus, these $g^{2}$ AR-graphs are edge-disjoint and form an $\mathrm{ODAR}(n,g)$. \vspace{-0.2cm}
%Next, we only need to prove that every edge $\{(x,a),(y,b)\}$ in $K_{n\times g}$ occurs exactly once in the collection $\{{A}_{i}:i\in[g]\}\bigcup \{{B}_{i,j}^{k}:i \in [g-1+\delta_{g}], j\in [\frac{g-\delta_{g}}{2}], k\in[2]\}$ and hence an $\mathrm{ODAR}(n,g)$ is constructed are required.\vspace{-0.2cm}
\begin{itemize}
 \item [$(1)$] When $a=b$, there exists an $r\in[n-1]$ such that $\{x,y\}\in {F}_{r}$ since $\mathcal{F}$ is a one-factorization of $K_{n}$.
 It follows that $\{(x,a),(y,b)\}\in {F}_{r}(a,a)\subseteq{A}_{i}$, where $i\in[g]$ and~$i+r\equiv a\pmod{g}$.\vspace{-0.2cm}
 \item [$(2)$] When $a< b$, there exists an $i\in[g-1+\delta_{g}]$ and $w\in[\frac{g-\delta_{g}}{2}]$ such that $\{a,b\}=\{a_{i,w},b_{i,w}\}\in {G}_{i}$ since $\mathcal{G}$ is a one-factorization or near one-factorization of $K_{g}$.
     There exists an $r\in[n]$ such that $(x,y)\in {D}_{r}$.
     If $r\in [\frac{n}{2}]$, then $\{(x,a),(y,b)\}\in {D}_{r}(a_{i,w},b_{i,w})\subseteq{B}_{i,j}^{1}$, where $j\in[\frac{g-\delta_{g}}{2}]$ and $j+r\equiv w\pmod{\frac{g-\delta_{g}}{2}}$.
     If $r\in[\frac{n}{2}+1, n]$, then $\{(x,a),(y,b)\}\in {D}_{r}(a_{i,w},b_{i,w})\subseteq{B}_{i,j}^{2}$, where $j\in[\frac{g-\delta_{g}}{2}]$ and $j+r\equiv w\pmod{\frac{g-\delta_{g}}{2}}$.
\end{itemize}\vspace{-0.2cm}
Then we obtained an $\mathrm{ODAR}(n,g)$.
\qed

\begin{Theorem}\label{ODAR_odd_n}
There exists an $\mathrm{ODAR}(n,g)$ for any positive integer $g$ and odd $n$ with $n\le g$.
\end{Theorem}
\proof As in the previous proof, let $\mathcal{G}=\{{G}_{1},{G}_{2},\ldots,{G}_{g-1+\delta_g}\}$ be a one-factorization or a near one-factorization  of $K_{g}$ over $[g]$, where ${G}_i$ is the same as (\ref{GGG}).

Let $\mathcal{F}=\{{F}_{1},{F}_{2},\ldots,{F}_{n}\}$ be a near one-factorization of the complete graph $K_{n}$ over $[n]$.
Since $n$ is odd and $n\le g$, we apply Strategy A with $s=n$ by (\ref{s}).
Fix $i\in[g]$.
For any $\{x,y\}\in F_{r}$, $r\in[n]$, let $c_{r,x}=c_{r,y}=i+r\in [g]\pmod{g}$.
Since $n\le g$, %the set $\{i+r:r\in[ n]\}$ contains $n$ distinct elements in $[g]$, where $i+r\in[g]\pmod{g}$. %is reduced modulo $g$ in $[g]$.
property (A) is satisfied.
Then, we construct an AR-graph
%Applying Strategy A with $s=n$ yields that every
$${A}_{i} = \bigcup_{r\in[ n]}{F}_{r}(i+r,i+r).$$
Therefore, we obtain $g$ AR-graphs $A_{i}$, $i\in[g]$.

Next, we construct further $g(g-1)$ edge-disjoint AR-graphs. Apply Strategy B with $t={n-1\over 2}$, as given by~(\ref{t}).
%with $t={n-1\over 2}$.
%Let $\mathcal{G}_{i}=\{\{a_{i,1},b_{i,1}\},$ $\{a_{i,2},b_{i,2}\},\ldots,\{a_{i,\frac{g-\delta_g}{2}},b_{i,\frac{g-\delta_g}{2}}\}\}$, where $a_{i,w}<b_{i,w}$, $1\leq w \leq \frac{g-\delta_g}{2}$.
%Evidently, $D_{1},D_{2}, \ldots,D_{{{n-1}\over 2}}$ and the collection $D_{{{n+1}\over 2}},D_{{{n+3}\over 2}}, \ldots,D_{n-1}$ are ${{n-1}\over 2}$ distinct sets satisfying property (B), respectively.
Evidently, $D_{1},D_{2}, \ldots,D_{{{n-1}\over 2}}$ are ${{n-1}\over 2}$ distinct sets satisfying property~(B).
For a fixed $i\in [g-1+\delta_g]$, since ${G}_{i}$ is a one-factor or near one-factor of $K_{g}$, for any $j\in[\frac{g-\delta_{g}}{2}]$, we construct an AR-graph
$${B}_{i,j}^{1}=\bigcup_{r \in [\frac{n-1}{2}]}{D}_{r}(a_{i,j+r},b_{i,j+r}),$$
where $j+r\in[\frac{g-\delta_{g}}{2}]\pmod{\frac{g-\delta_{g}}{2}}$.
And $D_{{{n+1}\over 2}},D_{{{n+3}\over 2}}, \ldots,D_{n-1}$ are also ${{n-1}\over 2}$ distinct sets satisfying property (B).
Thus, for any $i\in [g-1+\delta_g]$ and $j\in[\frac{g-\delta_{g}}{2}]$ , we construct an AR-graph
$${B}_{i,j}^{2}=\bigcup_{r\in[\frac{n+1}{2},n-1]}{D}_{r}(a_{i,j+r},b_{i,j+r}),$$
where $j+r\in[\frac{g-\delta_{g}}{2}]\pmod{\frac{g-\delta_{g}}{2}}$. %is reduced modulo $\frac{g-\delta_{g}}{2}$ to lie in  $[\frac{g-\delta_{g}}{2}]$.

Similar to the proof of Theorem \ref{ODAR_even_n}, we may show that the collection $\{{A}_{i}:i\in[g]\}\bigcup \{{B}_{i,j}^{k}:i \in [g-1+\delta_{g}], j\in[\frac{g-\delta_{g}}{2}], k\in[2]\}$ is an $\mathrm{ODAR}(n,g)$.
\qed

\begin{Example}\label{E3.1}
{\rm For $n=g=4$, take $F_1=\{\{1 ,2\},\{3,4\}\}$, $F_{2}=\{\{1,3\},\{2 ,4\}\}$, $F_{3}=\{\{1,4\},$ $\{2,3\}\}$ as a one-factorization of $K_{4}$.
 Let $s=n-1=3$.
For any $x\in [4]$, take $c_{r,x}=i+r\in[4]\pmod{4}$, $r\in[3]$ and hence $\{c_{r,x}:r\in[3]\}$ contains $3$ distinct elements in $[4]$. %the set
Then using Strategy A, we obtain the AR-graph
\begin{equation*}
\begin{split}
        A_{1} &=F_{1}(2,2)\cup F_{2}(3,3)\cup F_{3}(4,4) \\
        &=\{1_ 2\text{ }2_ 2,\text{ }3_ 2\text{ }4_ 2, \text{ }1_ 3\text{ }3_ 3,\text{ }2_ 3\text{ }4_ 3,\text{ }1_ 4\text{ }4_ 4,\text{ }2_ 4\text{ }3_ 4\},\\
\end{split}
\end{equation*}
where $F_{1}(2,2)=\{1_ 2\text{ }2_ 2,\text{ }3_ 2\text{ }4_ 2\}$, $F_{2}(3,3)=\{1_ 3\text{ }3_ 3,\text{ }2_ 3\text{ }4_ 3\}$ and $F_{3}(4,4)=\{1_ 4\text{ }4_ 4,\text{ }2_ 4\text{ }3_ 4\}$. Similarly, we obtain the following AR-graphs.
\begin{center}
\begin{tabular}{l l l l l l l l }
%$1_ 2\text{ }2_ 2$& $3_ 2\text{ }4_ 2$&
%$1_ 3\text{ }3_ 3$& $2_ 3\text{ }4_ 3$&
%$1_ 4\text{ }4_ 4$& $2_ 4\text{ }3_ 4$\\
$A_{2}=\{1_ 3\text{ }2_ 3,\text{ }3_ 3\text{ }4_ 3,\text{ }1_ 4\text{ }3_ 4,\text{ }2_ 4\text{ }4_ 4,\text{ }1_ 1\text{ }4_ 1,\text{ }2_ 1\text{ }3_ 1\}$,\\
%$1_ 3\text{ }2_ 3$& $3_ 3\text{ }4_ 3$&
%$1_ 4\text{ }3_ 4$& $2_ 4\text{ }4_ 4$&
%$1_ 1\text{ }4_ 1$& $2_1\text{ }3_ 1$\\
$A_{3}=\{1_ 4\text{ }2_ 4,\text{ }3_ 4\text{ }4_ 4,\text{ }1_ 1\text{ }3_ 1,\text{ }2_ 1\text{ }4_ 1,\text{ }1_ 2\text{ }4_ 2,\text{ }2_2\text{ }3_ 2\}$,\\
%$1_ 4\text{ }2_ 4$& $3_ 4\text{ }4_ 4$&
%$1_ 1\text{ }3_ 1$& $2_ 1\text{ }4_ 1$&
%$1_ 2\text{ }4_2$& $2_ 2\text{ }3_ 2$\\
$A_{4}=\{1_ 1\text{ }2_ 1,\text{ }3_ 1\text{ }4_ 1,\text{ }1_ 2\text{ }3_ 2,\text{ }2_ 2\text{ }4_ 2,\text{ }1_ 3\text{ }4_ 3,\text{ }2_ 3\text{ }3_ 3\}$.\\
%$1_ 1\text{ }2_ 1$& $3_ 1\text{ }4_ 1$&
%$1_ 2\text{ }3_ 2$& $2_ 2\text{ }4_ 2$&
%$1_ 3\text{ }4_ 3$& $2_ 3\text{ }3_ 3$\\
\end{tabular}
\end{center}

Let $t=\frac{n}{2}=2$,
$D_{1}=\{(1,2),(2,3),(3,4),(4,1)\}$, $D_{2}=\{(1,3),(2,4)\}$, $D_{3}=\{(1,4),(2,1),$ $(3,2),(4,3)\}$, and $D_{4}=\{(3,1),(4,2)\}$.
Take $G_i=F_{i}$, $i\in[3]$.
Here, in $G_{1}$, $a_{1,1}=1$, $b_{1,1}=2$, $a_{1,2}=3$ and $b_{1,2}=4$. Analogous notations apply to $G_2$ and $G_{3}$.
Then using Strategy B, we obtain the AR-graph
\begin{equation*}\begin{split}
B_{1,1}^{1}&=D_{1}(a_{1,2},b_{1,2})\cup D_{2}(a_{1,1},b_{1,1})\\
&=\{1_ 3\text{ }2_ 4, \text{ }2_ 3\text{ }3_ 4, \text{ }3_ 3\text{ }4_ 4, \text{ }4_ 3\text{ }1_ 4,\text{ }1_ 1\text{ }3_ 2, \text{ } 2_ 1\text{ }4_ 2\},\\
\end{split}
\end{equation*}
where $D_{1}(a_{1,2},b_{1,2})=D_{1}(3,4)=\{1_ 3\text{ }2_ 4, \text{ }2_ 3\text{ }3_ 4, \text{ }3_ 3\text{ }4_ 4, \text{ }4_ 3\text{ }1_ 4\}$ and $D_{2}(a_{1,1},b_{1,1})=D_{2}(1,2)=\{1_ 1\text{ }3_ 2, \text{ } 2_ 1\text{ }4_ 2\}$. Likewise,  we obtain the following AR-graphs.

$$B_{1,2}^{1}=D_{1}(1,2)\cup D_{2}(3,4) = \{1_ 1\text{ }2_ 2, \text{ }2_ 1\text{ }3_ 2, \text{ }3_ 1\text{ }4_ 2, \text{ }4_ 1\text{ }1_ 2,\text{ }1_ 3\text{ }3_ 4,\text{ } 2_ 3\text{ }4_ 4\},$$
$$B_{2,1}^{1}=D_{1}(2,4)\cup D_{2}(1,3)=\{1_ 2\text{ }2_ 4, \text{ }2_ 2\text{ }3_ 4, \text{ }3_ 2\text{ }4_ 4, \text{ }4_ 2\text{ }1_ 4,\text{ }1_ 1\text{ }3_ 3, \text{ }2_ 1\text{ }4_ 3\},$$
$$B_{2,2}^{1}=D_{1}(1,3)\cup D_{2}(2,4)=\{1_ 1\text{ }2_ 3, \text{ }2_ 1\text{ }3_ 3, \text{ }3_ 1\text{ }4_ 3, \text{ }4_ 1\text{ }1_ 3,\text{ }1_ 2\text{ }3_ 4, \text{ }2_ 2\text{ }4_ 4\},$$
$$B_{3,1}^{1}=D_{1}(2,3)\cup D_{2}(1,4)=\{1_ 2\text{ }2_ 3, \text{ }2_ 2\text{ }3_ 3, \text{ }3_ 2\text{ }4_ 3, \text{ }4_ 2\text{ }1_ 3,\text{ }1_ 1\text{ }3_ 4, \text{ }2_ 1\text{ }4_ 4\},$$
$$B_{3,2}^{1}=D_{1}(1,4)\cup D_{2}(2,3)=\{1_ 1\text{ }2_ 4, \text{ }2_ 1\text{ }3_ 4, \text{ }3_ 1\text{ }4_ 4, \text{ }4_ 1\text{ }1_ 4,\text{ }1_ 2\text{ }3_ 3, \text{ }2_ 2\text{ }4_ 3\},$$
$$B_{1,1}^{2}=D_{3}(3,4)\cup D_{4}(1,2)=\{1_ 3\text{ }4_ 4,\text{ } 2_ 3\text{ }1_ 4, \text{ }3_ 3\text{ }2_ 4, \text{ }4_ 3\text{ }3_ 4,\text{ }3_ 1\text{ }1_ 2, \text{ } 4_ 1\text{ }2_ 2\},$$
$$B_{1,2}^{2}=D_{3}(1,2)\cup D_{4}(3,4)=\{1_ 1\text{ }4_ 2,\text{ } 2_ 1\text{ }1_ 2, \text{ }3_ 1\text{ }2_ 2, \text{ }4_ 1\text{ }3_ 2,\text{ }3_ 3\text{ }1_ 4, \text{ } 4_ 3\text{ }2_ 4\},$$
$$B_{2,1}^{2}=D_{3}(2,4)\cup D_{4}(1,3)=\{1_ 2\text{ }4_ 4, \text{ }2_ 2\text{ }1_ 4, \text{ }3_ 2\text{ }2_ 4,\text{ } 4_ 2\text{ }3_ 4,\text{ }3_ 1\text{ }1_ 3, \text{ }4_ 1\text{ }2_ 3\},$$
$$B_{2,2}^{2}=D_{3}(1,3)\cup D_{4}(2,4)=\{1_ 1\text{ }4_ 3, \text{ }2_ 1\text{ }1_ 3, \text{ }3_ 1\text{ }2_ 3, \text{ }4_ 1\text{ }3_ 3,\text{ }3_ 2\text{ }1_ 4, \text{ } 4_ 2\text{ }2_ 4\},$$
$$B_{3,1}^{2}=D_{3}(2,3)\cup D_{4}(1,4)=\{1_ 2\text{ }4_ 3, \text{ }2_ 2\text{ }1_ 3, \text{ }3_ 2\text{ }2_ 3, \text{ }4_ 2\text{ }3_ 3,\text{ }3_ 1\text{ }1_ 4, \text{ } 4_ 1\text{ }2_ 4\},$$
$$B_{3,2}^{2}=D_{3}(1,4)\cup D_{4}(2,3)=\{1_ 1\text{ }4_ 4, \text{ }2_ 1\text{ }1_ 4, \text{ }3_ 1\text{ }2_ 4, \text{ }4_ 1\text{ }3_ 4,\text{ }3_ 2\text{ }1_ 3, \text{ }4_ 2\text{ }2_ 3\}.$$
%are 12 disjoint AR-graphs.
%Below we list 12 largest AR-graphs of $K_{4\times 4}$ by Strategy B.
 %$$D_{1}(1,2)\cup D_{2}(3,4);D_{1}(3,4)\cup D_{2}(1,2);D_{3}(1,2)\cup D_{4}(3,4);D_{3}(3,4)\cup D_{4}(1,2);$$
%$$D_{1}(1,3)\cup D_{2}(2,4);D_{1}(2,4)\cup D_{2}(1,3);D_{3}(1,3)\cup D_{4}(2,4);D_{3}(2,4)\cup D_{4}(1,3);$$
%$$D_{1}(1,4)\cup D_{2}(2,3);D_{1}(2,3)\cup D_{2}(1,4);D_{3}(1,4)\cup D_{4}(2,3);D_{3}(2,3)\cup D_{4}(1,4).$$

%Now, we construct the largest AR-graphs of $K_{4\times 4}$ by Strategy B.

Then, the above 4 AR-graphs by Strategy A and 12 AR-graphs by Strategy B together form an $\mathrm{ODAR}(4,4)$.
}\qed
\end{Example}

\section{Constructions for OF$(n,g)$s}

In view of Theorems \ref{ODAR_even_n} and \ref{ODAR_odd_n}, we are left to deal with optimal decompositions of $K_{n\times g}$ with $n>g$ and $gn$ even, i.e., one-factorizations of  $K_{n\times g}$ without parallel edges in each one-factor. We will display several direct constructions, a doubling construction, and a product construction.

%we begin with the constructions of ODAR$_3(n,g)$s with $n\le g$, for which  we will provide a complete solution.

\subsection{Constructions for $n=g+1$}

\begin{Lemma}\label{OF_n_odd}
%There exists a $\mathrm{TOC}_{n}(n,3,2)$ for any odd $n$.
There exists an $\mathrm{OF}(g+1,g)$ for any even $g$.
\end{Lemma}
\proof Let $g$ be even and $n=g+1$. For $j\in [n]$, define $${F}_{j}=\bigg\{\big\{j+w,j-w\big\}:w\in\left[\frac{n-1}{2}\right]\bigg\},$$
where $j+w,j-w\in[n]\pmod{n}$. Obviously, ${\cal F}=\{{F}_{1},{F}_{2},\ldots,{F}_{n}\}$ is a near one-factorization of $K_{n}$, and in ${F}_{j}$, vertex $j$ is isolated. On the other hand, there exists a one-factorization of $K_{g}$, say
${\cal G}=\{{G}_{1},{G}_{2},\ldots,{G}_{n-2}\}$.
We assign labels to each element in ${G}_i$ by
$${G}_{i}=\bigg\{\big\{a_{i,w},b_{i,w}\big\}:w\in \left[\frac{n-1}{2}\right],a_{i,w}<b_{i,w}\bigg\},\ i\in [n-2]=[g-1],$$ and particularly $${G}_{1}=\bigg\{\big\{w,n-w\big\}: w\in\left[\frac{n-1}{2}\right]\bigg\}.$$
%Assume w.l.o.g. that $a_{i,w}<b_{i,w}$ for $i\in[n-2]$ and $w\in[\frac{n-1}{2}]$.
Clearly $F_n=G_1$.

Note that
each one-factor of $K_{n\times g}$ has $\frac{gn}{2}=\frac{(n-1)n}{2}={n\choose 2}$ edges and a one-factorization of $K_{n\times g}$ consists of $g^{2}=(n-1)^{2}$  one-factors by Lemma \ref{d3}.
First, we construct $(n-1)(n-2)$ edge-disjoint one-factors without parallel edges.
For $i\in [n-2]$, $k\in [\frac{n-1}{2}]$, define two $n$-partite graphs with edge sets
$${C}_{i,k}^{1}=\bigg\{\big\{(j+w,a_{i,w+k}),(j-w,b_{i,w+k})\big\}:w\in \left[\frac{n-1}{2}\right],j\in [n]\bigg\}$$
and
$${C}_{i,k}^{2}=\bigg\{\big\{(j+w,b_{i,w+k}),(j-w,a_{i,w+k})\big\}:w\in \left[\frac{n-1}{2}\right],j\in [n]\bigg\},$$
where %the first components $j+w$ and $j-w$ are reduced by modulo $n$ to lie in $[n]$ and
the subscript $w+k\in [\frac{n-1}{2}]\pmod{\frac{n-1}{2}}$.
%$w+k$ is reduced by modulo $\frac{n-1}{2}$ to lie in $[\frac{n-1}{2}]$.
Fix $i\in [n-2]$, $k\in [\frac{n-1}{2}]$ and $l\in [2]$.
Since ${\cal F}$ is a near
one-factorization of $K_{n}$, then ${C}_{i,k}^{l}$ has $n\choose 2$ edges and does not contain parallel edges.
For any vertex $(x,a)\in [n]\times [g]$, there exists a $w\in[\frac{n-1}{2}]$ such that $a_{i,w+k}=a$ or $b_{i,w+k}=a$ since ${G}_{i}$ is a one-factor of $K_{g}$.
The set $\{j+w:j\in[n]\}=\{j-w:j\in[n]\}=[n]$;
hence, there exists a $j\in [n]$ such that $j+w=x$ or $j-w=x$.
Then, there exists a unique edge in ${C}_{i,k}^{l}$ that contains  $(x,a)$.
Therefore, ${C}_{i,k}^{l}$ is a one-factor of $K_{n\times g}$ without parallel edges.
Thus far,
for all edges $\{(x,a),(y,b)\}$ in $K_{n\times g}$ with $a\neq b$, there exist $w\in[\frac{n-1}{2}]$ and $j\in[n]$ such that $\{x,y\} =\{j+w,j-w\}$  and then we have some $i\in[n-2]$ and $k\in[\frac{n-1}{2}]$ such that $(a,b) =(a_{i,w+k},b_{i,w+k})$ if $a<b$ or $(a,b) =(b_{i,w+k},a_{i,w+k})$ if $a>b$, due to the near one-factorization ${\cal F}$ and the one-factorization ${\cal G}$.
There exist $g(g-1){n\choose 2}$ edges of the form $\{(x,a),(y,b)\}$ in $K_{n\times g}$ with $a\neq b$.
%It then follows that all such edges can be found in the one-factors constructed above.
It then follows that all such edges can be found exactly once in the one-factors constructed above.

Next, we construct another $n$ one-factors of $K_{n\times g}$ without parallel edges.
For $r \in[n]$, define
\begin{align*}
{E}_{r}&=\bigg\{\big\{(r+j+w,j),(r+j-w,j)\big\}:w\in\left[\frac{n-1}{2}\right],j \in[g]=[n-1]\bigg\}\\
&\quad \ \ \bigcup
\bigg\{\big\{(r+w,w),(r-w,n-w)\big\}:w\in\left[\frac{n-1}{2}\right]\bigg\}.
\end{align*}
%where all the first components lie in $[n]$ by modulo $n$.
Fix $r \in[n]$. Note that we also used the near one-factorization ${\cal F}$ as above. Thus, ${E}_{r}$ has $n\choose 2$ edges and does not contain parallel edges. For any vertex $(x,a)\in [n]\times [g]$, $\big\{\{r+a+w,r+a-w\}:w\in[\frac{n-1}{2}]\big\}$
 is a near one-factor of $K_{n}$, and vertex $r+a$ is isolated.
Hence, there exists a unique edge in the set $\big\{\{(r+a+w,a),(r+a-w,a)\}:w\in[\frac{n-1}{2}]\big\}$ that contains the vertex $(x,a)$ if $x\not\equiv r+a\pmod n$.
For any vertex $(r+a,a)$,
there exists a unique edge in $\big\{\{(r+w,w),(r-w,n-w)\}:w\in[\frac{n-1}{2}]\big\}$ that contains it. It follows that ${E}_{r}$ is a one-factor of $K_{n\times g}$ without parallel edges.

It is straightforward to show that for all edges $\{(x,a),(y,a)\}$ in $K_{n\times g}$, each can be found in the set $\big\{\{(r+w+a,a),(r-w+a,a)\}:w\in\left[\frac{n-1}{2}\right],r\in[n]\big\}$. Finally, we should remove a one-factor to delete all repeated edges we used.
In fact,  $$\left\{\big\{(r+w,w),(r-w,n-w)\big\}:w\in\left[\frac{n-1}{2}\right],r\in[n]\right\}={C}_{1,\frac{n-1}{2}}^{1}.$$ Therefore, the collection $$\left\{{C}_{i,k}^{l}:i\in [n-2], k\in\left[\frac{n-1}{2}\right],l\in[2],(i,k,l)\neq\big(1,\frac{n-1}{2},1\big)\right\}\bigcup\big\{{E}_{r}:r\in [n]\big\}$$ is an $\mathrm{OF}(g+1,g)$.
\qed

For the case when $g$ is odd and $n=g+1$, we also provide a direct construction. Prior to the construction, we require two notions.
An {\em orthogonal array} $\mathrm{OA}(t,n,q)$ is a $q^{t}\times n$ array $A$ over an alphabet $Q$ with $q$ symbols such that each ordered $t$-tuple of $Q$ appears in exactly one row in each projection of $t$ columns from $A$.
The following results \cite[pp. 317--331]{code-impo1} are derived from $\mathrm{MDS}$ codes.
\begin{Theorem}[\label{OA_B4}{\cite{code-impo1}}]
If $q$ is a prime power and $t$ is a positive integer with $t\leq q-1$, then there exists an $\mathrm{OA}(t,n,q)$ for any integer $n$ with $t\leq n\leq q+1$. Moreover, if $q$ is an even prime power and $t\in \{3,q-1\}$, then there exists an $\mathrm{OA}(t,n,q)$ for any $n$ with $t\leq n\leq q+2$.
\end{Theorem}

A {\em{Latin square}} of order $n$, denoted by $\mathrm{LS}(n)$, is an $n\times n$ array in which each cell contains a single symbol from an $n$-set~$X$, such that each symbol of $X$ occurs exactly once in each row and exactly once in each column. Suppose that $L=(a_{i,j})$ and $L^{'} = (b_{i,j})$ are two $\mathrm{LS}(n)$s on~$X$.~$L$ and $L^{'}$ are {\em{orthogonal}} if every element of $X\times X$ occurs exactly once among the $n^{2}$ pairs $(a_{i,j},b_{i,j})$, $i,j\in X$.

\begin{Theorem}[\label{DLS1}\cite{OLS}]
There exists a pair of orthogonal Latin squares of order $n$ for any positive integer $n \neq 2,6$.
\end{Theorem}

\begin{Lemma}\label{OF_n_even}
%There exists a $\mathrm{TOC}_{q+1}(q+1,3,2)$ for any odd prime power $q$.
There exists an $\mathrm{OF}(q+1,q)$ for any odd prime power $q$.
\end{Lemma}
\proof For an odd prime power $q$, let ${M}$ be an $\mathrm{OA}(2,q+1,q)$ over $[q]$, which exists by Theorem~\ref{OA_B4}.
Further assume that the symbols in the first coordinate of the first $q$ rows of $M$ are ones, the symbols in the first coordinate of the next $q$ rows of ${M}$ are twos, and so on, until the symbols in the first coordinate of the last $q$ rows of $M$ are $q$'s. Thus, $M$ has the form
$${M}=\left(\begin{matrix}
M_{1} \\ M_{2} \\ \vdots \\ M_{q}\\
\end{matrix}\right),\ \text {where}\
M_{i}=\left(\begin{matrix}
i & m_{1,1}^{(i)} & m_{1,2}^{(i)} & \cdots & m_{1,q}^{(i)}\\
i & m_{2,1}^{(i)} & m_{2,2}^{(i)} & \cdots & m_{2,q}^{(i)}\\
\vdots & \vdots & \vdots & & \vdots\\
i & m_{q,1}^{(i)} & m_{q,2}^{(i)} & \cdots & m_{q,q}^{(i)}\\
\end{matrix}\right). $$
Then all columns of $M_{i}$ except the first column form permutations of $1,2,\ldots,q$.

Let $L=(l_{i,j})$ be a Latin square over $[q]$.
Since $q+1$ is even, we take a one-factorization of $K_{q+1}$ with vertex set $[q+1]$ and denote the one-factors by ${F}_{1},{F}_{2},\ldots,{F}_{q}$. Assume that $\{r,q+1\}\subseteq {F}_{r}$, $r\in [q]$.
%Fix $i,j$ with $i,j\in[q]$.
%Apply Strategy A with $s=q$. For any $\{x,y\}\in F_{r}$ with $q+1\not\in\{x,y\}$, $r\in[q]$, we choose $c_{r,x}^{i,j}=m_{j+r-1,x}^{(i)}$ and $c_{r,y}^{i,j}=m_{j+r-1,x}^{(i)}$.
%For any $\{x,y\}\in F_{r}$, $r\in[n-1]$, we choose $c_{r,x}^{i}=c_{r,y}^{i}=i+r\in [g]$.
%Since $n\le g$, the set $\{i+r:r\in[n-1]\}$ contains $n-1$ distinct nonzero elements in $[g]$, where $i+r$ is reduced modulo $g$ in $[g]$, the property (A) is satisfied.
Apply Strategy A with $s=q$ by (\ref{s}).
Then for any $i,j \in [q]$, define a sub-graph of $K_{n\times g}$ by
$${C}_{i,j}=\bigcup_{r\in [q]}{A}_{i,j}^{r},i,j \in [q],$$
where
$${A}_{i,j}^{r}=\left\{\big\{(x,m_{j+r-1,x}^{(i)}),(y,m_{j+r-1,y}^{(i)})\big\}:q+1\notin \{x,y\}\in {F}_{r}\right\}\bigcup \left\{\big\{(r,m_{j+r-1,r}^{(i)}),(q+1,l_{r,i})\big\}\right\},
$$
and $j+r-1\in [q]\pmod{q}$.
Since $\{m_{j+r-1,z}^{(i)}:r\in [q]\}=\{l_{r,i}:r\in [q]\}=[q]$, property (A) is satisfied.
Hence, ${C}_{i,j}$ is a one-factors of $K_{(q+1)\times q}$ without parallel edges.

Note that we have obtained $q^{2}$ one-factors of $K_{q+1}$, by Lemma \ref{d3},
we only need prove that every edge $\{(x,a),(y,b)\}$ of $K_{(q+1)\times q}$ occurs at least once in the collection $\{{C}_{i,j}:i,j\in [q]\}$. Thus, these $q^{2}$ one-factors are edge-disjoint and form an  $\mathrm{OF}(q+1,q)$.
\vspace{-0.2cm}
\begin{itemize}\vspace{-0.2cm}
\item[(1)] If $\{x,y\}\subseteq[q]$ and $\{x,y\}\in {F}_{r}$, %For any $1\leq a,b \leq q$,
there exist $i,j\in[q]$ such that $m_{j+r-1,x}^{(i)} =a$ and $m_{j+r-1,y}^{(i)}=~b$ since $M$ is an $\mathrm{OA}(2,q+1,q)$ over $[q]$.
    So $\{(x,a),(y,b)\}$ occurs at least once in ${A}_{i,j}^{r}\subseteq{C}_{i,j}$.\vspace{-0.2cm}
\item[(2)] If $\{(x,a),(y,b)\} = \{(r,a),(q+1,b)\}$, there exists an $i\in[q]$ such that $l_{r,i} = b$ as $L$ is a Latin square over $[q]$. And there exists a $j\in[q]$ such that $m_{j+r-1,r}^{(i)}= a$ as the $r$-th column of $M_{i}$ is a permutation of $1,2,\ldots,q$. So $\{(r,a),(q+1,b)\}$ occurs at least once in ${A}_{i,j}^{r}\subseteq{C}_{i,j}$.\vspace{-0.2cm}
\end{itemize}\vspace{-0.2cm}
%So the collection $\{{C}_{i,j}:i,j\in [q]\}$ is an $\mathrm{OF}(q+1,q)$.
\qed

\subsection{Doubling construction}

In order to determine the existence of $\mathrm{OF}(n,g)$s recursively, we present a doubling construction.
\begin{Lemma}\label{OF_double}
Let $n,g$ be positive integers satisfying $ng$ is even and $n>g$.
If there exists an $\mathrm{OF}(n,g)$, then there exists an $\mathrm{OF}(2n,g)$.
\end{Lemma}
\proof Suppose that $\{A_{i}:i\in[g(n-1)]\}$ is the given $\mathrm{OF}(n,g)$, where each  $A_{i}$ is a one-factor of~$K_{n\times g}$ without parallel edges.
Firstly, consider all edges in $\{\{(x,a),(y,b)\}\in K_{2n\times g}:\{x,y\}\subseteq~[n] \text{ or } \{x,y\}\subseteq [n+1,2n]\}$. For any $i\in [g(n-1)]$, define
$$B_{i} = \big\{\{(x,a),(y,b)\},\{(x+n,a),(y+n,b)\}:\{(x,a),(y,b)\}\in A_{i}\big\}.$$
Since $A_{i}$ is a one-factor of $K_{n\times g}$ without parallel edges, which implies that $B_{i}$ is a one-factor of~$K_{2n\times g}$ without parallel edges.
For each edge $\{(x,a),(y,b)\}$ with $\{x,y\}\subseteq [n] $ or  $\{x,y\}\subseteq [n+1,2n]\}$ in $K_{2n\times g}$,
there exists an $i\in [g(n-1)]$, such that $\{(x,a),(y,b)\}\in {B}_{i}$ as $\{{A}_{i}:i\in [g(n-1)]\}$ is a one-factorization of $K_{n\times g}$.
Simple counting yields that we obtain $g(n-1)$ edge-disjoint one-factors of $K_{n\times 2g}$. %and all such edges of the form $\{(x,a),(y,b)\in K_{2n\times g}:\{x,y\}\subseteq [n] \text{ or } \{x,y\}\subseteq [n+1,2n]\}$, each can be found exactly once in the set $\{B_{i}:i\in g(n-1)\}$.
To complete an $\mathrm{OF}(2n,g)$, we still need $gn$ one-factors of~$K_{2n\times g}$ by Lemma~\ref{d3}.

Secondly, consider all edges of the form $\{(x,a),(y,b)\}$ with $\{x,y\}\not \subseteq [n]$ and $\{x-n,y-n\}\not\subseteq~[n]$ in $K_{2n\times g}$. For $r\in [n]$, define
$$F_{r}=\big\{\{x,x+r+n\}:x\in [n]\big\},$$
where $x+r+n\in[n+1,2n]\pmod{n}$.
%$x+r+n$ is reduced modulo $n$ to lie in $[n+1,2n]$.
Then $F_{r}$, $r\in [n]$, are  edge-disjoint one-factors of $K_{2n}$. %We consider two cases.
We distinguish two cases
depending on the value of $g$.

\textbf{Case 1:}
$g\neq 2,6$.

Let $M_{1}=(m_{i,j}^{(1)})$ and $M_{2}=(m_{i,j}^{(2)})$ be a pair of orthogonal Latin squares over $[g]$, which exists by Theorem \ref{DLS1}.
For any $i\in [g]$, $k\in[n]$, define
$$C_{i,k}=\bigcup_{j\in[g]}F_{k+j}\left(m_{i,j}^{(1)},m_{i,j}^{(2)}\right),$$
where $k+j\in[n]\pmod{n}$. %is reduced modulo $n$ to an element in $[n]$.
%Fix $i,k$ with $1\leq i  \leq g$ and $0\leq k  \leq n-1$.
Since $n>g$, $F_{k+1}, F_{k+2},\ldots, $ $F_{k+g}$ are pairwise disjoint one-factors of $K_{2n}$.
Furthermore, since the sets $\{m_{i,j}^{(1)}:j\in [g]\}=\{m_{i,j}^{(2)}:j\in[g]\}=[g]$,  property (A) is satisfied. Thus each ${C}_{i,k}$ is a one-factor of $K_{2n\times g}$ without parallel edges by Strategy A.

For any edge $\{(x,a),(y,b)\}$ with $\{x,y\}\not \subseteq [n]$ and $\{x-n,y-n\}\not\subseteq [n]$ in $K_{2n\times g}$, we let $r\in [n]$ such that $y-x\equiv r\pmod{n}$.
Then $\{x,y\}\in {F}_{r}$. For $a,b \in [g]$, there exist~$i$ and $j$ such that $m_{i,j}^{(1)}=a$ and $m_{i,j}^{(2)}=b$ as $M_{1}$ and $M_{2}$ are a pair of orthogonal Latin squares over $[g]$. Thus, the edge $\{(x,a),(y,b)\}\in {F}_{r}(m_{i,j}^{(1)},m_{i,j}^{(2)}) \subseteq {C}_{i,k}$, where $k\in [n]$ and $k+j\equiv r\pmod{n}$.
It follows that these ${C}_{i,k}$, $i\in [g]$, $k\in [n]$, are $gn$ edge-disjoint  one-factors of $K_{2n\times g}$ without parallel edges.
Therefore,  $$\{{B}_{i}:i\in[g(n-1)]\}\bigcup\{{C}_{i,k}:i\in [g],k\in [n]\}$$ forms an $\mathrm{OF}(2n,g)$.

%Next, we prove that $\{B_{i}:i\in [g(n-1)]\}\bigcup\{{C}_{i,k}:i\in [g],k\in[n]\}$ is a one-factorization of $K_{2n\times g}$. Let $\{(x,a),(y,b)\}$ be an edge in $K_{2n\times g}$ and let $x<y$.\vspace{-0.2cm}
%\begin{itemize}
%\item[(1)] If $\{x,y\}\subseteq [n]$ or $\{x-n,y-n\}\subseteq [n]$, we can find $i\in [g(n-1)]$, such that $\{(x,a),(y,b)\}\in {B}_{i}$ as $\{{A}_{i}:i\in [g(n-1)]\}$ is a one-factorization of $K_{n\times g}$. \vspace{-0.2cm}
%\item[(2)] If $\{x,y\}\not \subseteq [n]$ and $\{x-n,y-n\}\not\subseteq [n]$, letting $r\in [n]$ and $y-x\equiv r\pmod{n}$. Then we have $\{x,y\}\in {F}_{r}$. For $a,b \in [g]$, we can find $i$ and $j$ such that $m_{i,j}^{(1)}=a$ and $m_{i,j}^{(2)}=b$ as $M_{1}$ and $M_{2}$ are a pair of orthogonal Latin squares over $[g]$. Thus, the edge $\{(x,a),(y,b)\}\in {F}_{r}(m_{i,j}^{(1)},m_{i,j}^{(2)}) \subseteq {C}_{i,k}$, where $k\in [n]$ and $k+j\equiv r\pmod{n}$.
%\end{itemize}\vspace{-0.2cm}

\textbf{Case 2:} $g=2,6$.

We apply Strategy A and let $s=g$ by (\ref{s}). We have that each
$$E_{i}=\bigcup_{a\in [g]}{F}_{i+a}(a,a),\ i\in[n],$$
is a one-factor of $K_{2n\times g}$ without parallel edges,
where $i+a\in[n]\pmod{n}$.
%$i+a$ is reduced modulo $n$ to lie in $[n]$.
Again note that, for any $i\in[n]$, the one-factors  $F_{i+1},F_{i+2},\ldots,F_{i+g}$ are pairwise disjoint as $n>g$.

Since $g=2,6$, we take a one-factorization of $K_{g}$, say
$\mathcal{G}=\{G_{1},G_{2},\ldots,G_{g-1}\}$. Assign labels to each element in ${G}_i$ by
$${G}_{i}=\left\{\{a_{i,w},b_{i,w}\}:w\in\left[\frac{g}{2}\right], a_{i,w}<b_{i,w}\right\}.$$
%Assume w.l.o.g. that $a_{i,w}<b_{i,w}$ for $i\in[g-1]$ and $w\in[\frac{g}{2}]$.
%where ${G}_{i}=\big\{\{a_{i,1},b_{i,1}\},\{a_{i,2},b_{i,2}\},\ldots,\{a_{i,\frac{g}{2}},b_{i,\frac{g}{2}}\}\big\}$ and $a_{i,w}<b_{i,w}$($i\in[g-1],w\in[{g\over 2}]$).
%Fix $i$ with $1 \leq i\leq g-1$ and $j$ with $0\leq j \leq n-1$.
%For any $j\in [n]$,
%the edge sets $F_{j+1}, F_{i+2},\ldots,\ldots,F_{j+g}$ are pairwise disjoint one-factors of $K_{2n}$ since $n>g$, where $j+w$ with $w\in[g]$ is reduced modulo $n$ to lie in $[n]$.
Fix $i,j$ with $i\in [g-1]$ and $j\in [n]$ and apply Strategy A with $s=g$.
Since $n>g$, $F_{j+1}, F_{j+2},\ldots, $ $F_{j+g}$ are pairwise disjoint one-factors of $K_{2n}$.
Define a graph with edge set
$${I}_{i,j} = \left(\bigcup_{w\in [\frac{g}{2}]}{F}_{j+w}\big(a_{i,w},b_{i,w}\big)\right)\bigcup \left(\bigcup_{w\in [\frac{g}{2}]}{F}_{j+w+\frac{g}{2}}\big(b_{i,w},a_{i,w}\big)
\right),$$
where $j+w,j+w+\frac{g}{2}\in[n]\pmod{n}$. %by taking modulo $n$.
Note that one of the two vertices of each edge in~$F_{r}$, $r\in [n]$, is contained in $[n]$ and the other in $[n+1,2n]$.
For any $i\in [g-1]$, the set $\{a_{i,w},b_{i,w}:w \in [\frac{g}{2}]\}=[g]$, so property (A) is satisfied.
Hence, each ${I}_{i,j}$
is a one-factor of $K_{2n\times g}$ without parallel edges.

For any edge $\{(x,a),(y,b)\}$ with $\{x,y\}\not \subseteq [n]$ and $\{x-n,y-n\}\not\subseteq [n]$ in $K_{2n\times g}$,
we let $r\in [n]$ such that $y-x\equiv r\pmod{n}$. Then $\{x,y\}\in {F}_{r}$.
If $a=b$, the edge $\{(x,a),(y,b)\}\in {F}_{r}(a,a)\subseteq {E}_{i}$, where $i\in [n]$ and $i+a\equiv r\pmod{n}$.
If $a\neq b$, there exist $i\in [g-1]$ and $w\in [\frac{g}{2}]$ such that $\{a_{i,w},b_{i,w}\}=\{a,b\}$ since $\mathcal{G}$ is a one-factorization of $K_{g}$.
Then  $\{(x,a),(y,b)\} \in {F}_{r}(a_{i,w},b_{i,w})\subseteq{I}_{i,j}$, where $j\in [n]$ and $j+w\equiv r\pmod{n}$ if $a<b$, or $\{(x,a),(y,b)\} \in {F}_{r}(b_{i,w},a_{i,w})\subseteq{I}_{i,j}$ if $a>b$, where $j\in [n]$ and $j+w+\frac{g}{2}\equiv r\pmod{n}$. It follows that the collection $\{E_{i}:i\in [n]\}\bigcup\{{I}_{i,j}:i\in [g-1],j\in [n]\}$ contains $gn$ edge-disjoint one-factors of $K_{2n\times g}$ without parallel edges. Thus, $$\big\{{B}_{i}:i\in  [g(n-1)]\big\}\bigcup\big\{E_{i}:i\in [n]\big\}\bigcup\big\{{I}_{i,j}:i\in [g-1],j\in [n]\big\}$$ is a one-factorization of $K_{2n\times g}$.
So, an $\mathrm{OF}(2n,g)$ is produced and the proof is complete.
\qed

%Next, we prove that  $$\big\{{B}_{i}:i\in  [g(n-1)]\big\}\bigcup\big\{E_{i}:i\in [n]\big\}\bigcup\big\{{I}_{i,j}:i\in [g-1],j\in [n]\big\}$$ is a one-factorization of $K_{2n\times g}$. Let $\{(x,a),(y,b)\}$ be an edge in $K_{2n\times g}$ with $x<y$.
%We only need to cheek that the case with $\{x,y\}\not \subseteq [n]$ and $\{x-n,y-n\}\not\subseteq [n]$.
%Let $r\in [n]$ and $y-x\equiv r\pmod{n}$. Then we have $\{x,y\}\in {F}_{r}$. If $a=b$, the edge $\{(x,a),(y,b)\}\in {F}_{r}(a,a)\subseteq {E}_{i}$, where $i\in [n]$ and $i+a\equiv r\pmod{n}$. If $a\neq b$, we can find $i\in [g-1]$ and $w\in [\frac{g}{2}]$ such that $\{a_{i,w},b_{i,w}\}=\{a,b\}$ since $\{{G}_{1},{G}_{2},\ldots,{G}_{g-1}\}$ is a one-factorization of $K_{g}$. Then  $\{(x,a),(y,b)\} \in {F}_{r}(a_{i,w},b_{i,w})\subseteq{I}_{i,j}$, where $j\in [n]$ and $j+w\equiv r\pmod{n}$ if $a<b$, or $\{(x,a),(y,b)\} \in {F}_{r}(b_{i,w},a_{i,w})\subseteq{I}_{i,j}$ if $a>b$, where $j\in [n]$ and $j+w+\frac{g}{2}\equiv r\pmod{n}$.

\subsection{Product construction}
%In this subsection, we present a product construction.

\begin{Lemma}\label{OF-mul-construction}
Let $u,v,g$ be positive integers, where $gu$ and $gv$ are even and $u,v> g$. If there exist an $\mathrm{OF}(u,g)$ and an $\mathrm{OF}(v,g)$, then there exists an $\mathrm{OF}(uv,g)$.
\end{Lemma}
\proof The desired $\mathrm{OF}(uv,g)$ will be constructed for the multipartite graph $K_{uv\times g}$ with vertex set $[u]\times[v]\times [g]$ and  $uv$  parts $\{x\}\times \{\alpha\}\times [g]$ with $x\in [u]$ and $\alpha\in[v]$.
According to the assumption, we have an $\mathrm{OF}(v,g)$ on vertex set $[v]\times [g]$, say $\{{A}_{i}:i\in[g(v-1)]\}$, where each ${A}_{i}$ is a one-factor of $K_{v\times g}$ without parallel edges.
%On $\mathcal{A}_{i}$, add a new component to every vertex in the original graph.
We construct a new edge set ${A}_{x,i}$ with $x\in [u]$ by adjoining a new component $x$ to each vertex of ${A}_{i}$, that is
$${A}_{x,i}=\big\{\{(x,\alpha,a),(x,\beta,b)\}:
\{(\alpha,a),(\beta,b)\}\in {A}_{i}\big\}.$$
It is immediate that, for every $i\in[g(v-1)]$,
${C}_{i}:=\bigcup_{x \in [u]}{A}_{x,i}$
is a one-factor of $K_{uv\times g}$ without parallel edges.
Now, we run out of all edges of the form $\{(x,\alpha,a),(x,\beta,b)\}$, where $x\in[u]$, $\alpha \neq \beta \in[v]$ and $a,b\in[g]$, as ${A}_{i}$, $i\in g(v-1)$ form a one-factorization of $K_{v\times g}$ over $[v]\times[g]$.
%on vertex set $[u]\times [v]\times[g]$. So we obtain $g(v-1)$ disjoint one-factors of $K_{uv\times g}$. Obviously, $\mathcal{E}_{i}$ has $\Sigma_{x\in[u]}|\mathcal{A}_{x,i}|=\frac{uv(q-1)}{2}$ codewords. So, we have $g(v-1)$ optimal $(uv,3,2)_{q}$-codes.

Suppose that $\{{B}_{i}:i\in[g(u-1)]\}$ is the given $\mathrm{OF}(u,g)$ on vertex set $[u]\times [g]$, where each~${B}_{i}$ is a one-factor of $K_{u\times g}$ without parallel edges.
Let $L=(l_{i,j})_{v\times v}$ be a Latin square over $[v]$.
Hence, for any $x,y\in [u]$ with $x\neq y$ and any $s\in [v]$, $\big\{\{(x,j),(y,l_{s,j})\}:j\in [v]\big\}$ is a one-factor of the bipartite graph with parts $\{x\}\times [v]$ and $\{y\}\times [v]$.
Further, for any $i\in [g(u-1)]$ and any $e=\{(x,a),(y,b)\}\in{B}_{i}$ with $x< y$, define
$${E}_{e,i}^{s}=\big\{\{(x,j,a),(y,l_{s,j},b)\}:j \in [v]\big\}.$$
Then, for any $i\in [g(u-1)]$ and $s\in [v]$, define
$${E}_{i}^{s}=\bigcup_{e\in{B}_{i}}{E}_{e,i}^{s}.$$
The number of edges in ${E}_{i}^{s}$ is $\frac{uvg}{2}$ clearly.
Then we prove that each  ${E}_{i}^{s}$ is a one-factor of $K_{uv\times g}$  without parallel edges.
We use proof by contradiction.
Assume that there exists a pair of parallel edges $e_{1}=\{(x,\alpha,a_{1}),(y,\beta,b_{1})\}$ and $e_{2}=\{(x,\alpha,a_{2}),(y,\beta,b_{2})\}$ in  ${E}_{i}^{s}$. Then $\{(x,a_{1}),(y,b_{1})\},\{(x,a_{2}),(y,b_{2})\}$ are a pair of parallel edges in  ${B}_{i}$, a contradiction.
For any vertex $(x,\alpha,a)\in [u]\times [v]\times [g]$, there exists a unique edge $e\in{B}_{i}$ containing $(x,a)$, since ${B}_{i}$ is a one-factor of $K_{u\times g}$;
and there is an edge $e'\in {E}_{e,i}^{s}$ such that $(x,\alpha,a)\in e'$ as $L$ is a Latin square over $[v]$.
Hence each  ${E}_{i}^{s}$ forms a one-factor of $K_{uv\times g}$ without parallel edges.

Now, we prove that the collection $$\big\{{C}_{i}:i\in [g(v-1)]\big\}\bigcup \big\{{E}_{i}^{s}:i\in [g(u-1)],s\in [v]\big\}$$
is a one-factorization of $K_{uv\times g}$.
Note that we have obtained $uv(g-1)$ one-factors of $K_{uv\times g}$. By Lemma \ref{d3}, we only need to prove that every edge $\{(x,\alpha,a),(y,\beta,b)\}$ in
$K_{uv\times g}$ occurs at least once.\vspace{-0.2cm}
%Let $\{(x,\alpha,a),(y,\beta,b)\}$ be an edge in $K_{uv\times g}$.\vspace{-0.2cm}
\begin{itemize}
 \item [$(1)$] If $x=y$, then $\alpha \neq \beta $ and there exists an $i\in [g(v-1)]$ such that $\{(x,\alpha,a),(x,\beta,b)\}\in {A}_{x,i}\subseteq{C}_{i}$ since $\{{A}_{i}:i\in [g(v-1)]\}$ forms an $\mathrm{OF}(v,g)$. \vspace{-0.2cm}
 \item [$(2)$] If $x< y$, there exists an $i\in [g(u-1)]$ such that there exists an edge $e = \{(x,a),(y,b)\}\in {B}_{i}$ since  $\{B_{i}:i\in [g(u-1)]\}$ is a one-factorization of $K_{u\times g}$ on vertex set $[u]\times [g]$. Next there exists an $s\in[v]$ such that $l_{s,\alpha} = \beta$ as $L$ is a Latin square over $[v]$. Then the edge $\{(x,\alpha,a),(x,\beta,b)\}\in {E}_{e,i}^{s}\subseteq{E}_{i}^{s}$.
\end{itemize}\vspace{-0.2cm}
Hence, the desired $\mathrm{OF}(uv,g)$ has been produced.
\qed

\subsection{Construction for OF$(n,g)$ with odd $n$}

%In this subsection we will construct $\mathrm{OF}(n,g)$s for relatively large $n$ in four propositions.
%In this subsection we will construct $\mathrm{OF}(n,g)$s for relatively large $n$.

\begin{Proposition}\label{OF_g_even_1}
There exists an $\mathrm{OF}(n,g)$ for any even $g$ and odd $n$ with $n\geq \mathop{\max}\{2g-1,g+3\}$.
\end{Proposition}
\proof Let $n$ be odd, $g$ be even, and $n\geq \mathop{\max}\{2g-1,g+3\}$.
For $r\in [n]$, define
$${F}_{r}=\left\{\{r+i,r-i\}:i\in \left[\frac{n-1}{2}\right]\right\},$$
where $r+i,r-i\in[n]\pmod{n}$.
%$r+i$ and $r-i$ are reduced modulo $n$ to lie in $[n]$.
Then  ${\cal F}=\{{F}_{1},{F}_{2},\ldots,{F}_{n}\}$ forms a near one-factorization of~$K_{n}$, such that in ${F}_{r}$, vertex $r$ is isolated.
Our construction proceeds in two steps.

\textbf{Step 1:} %Applying Strategy A with $s=g$, we have that $\mathcal{C}_{i}=\bigcup_{1\leq a \leq g}\mathcal{F}_{i+a-1}(a,a)$ ($1\leq i \leq n$), has $\frac{g(n-1)}{2}$ edges and it is an almost $1$-regular subgraph of $K_{2n\times g}$ without parallel edges.In order to construct a one-factor of $K_{2n\times g}$, we need look for $\frac{g}{2}$ edges again.
For any $i \in [n]$, define
$${A}_{i} = \left(\bigcup_{a\in [g]}{F}_{i+a-1}(a,a)\right)\bigcup\big\{\{(i+a-1,a),(i+a,a+1)\}:a\text{ is odd, } a<g\big\},$$
where $i+a-1,i+a\in[n]\pmod{n}$ and $a+1\in[g]\pmod{g}$. %$i+a-1,i+a\in[n]$ by modulo $n$ and $a+1\in[g]$ by modulo $g$.
Note that $g\leq\frac{n+1}{2}$ and $\{x,x+1\}\in {F}_{x+\frac{n+1}{2}}$, where $x+1,x+\frac{n+1}{2}\in [n]\pmod{n}$. % by modulo $n$.
Then  $F_{i+a-1}$, $a\in [g]$, and $\{\{i+a-1,i+a\}:~a\text{ is odd, } a<g\}$ are mutually disjoint. So each  ${A}_{i}$ does not contain parallel edges. For any vertex $(x,a)$ in $K_{n\times g}$, there exists a unique edge $e$ in  $F_{i+a-1}(a,a)$ that contains $(x,a)$ with $x\not\equiv i+a-1\pmod n$; and there exists a unique edge in  $\{\{(i+a-1,a),(i+a,a+1)\}:~a\text{ is odd, } a<g\big\}$ that contains the vertex $(i+a-1,a)$.
Therefore, each  ${A}_{i}$ forms a one-factor of $K_{n\times g}$ without parallel edges and we obtain $n$ one-factors of $K_{n\times g}$ without parallel edges.

\textbf{Step 2:} Since $g$ is even,
there exists a one-factorization of $K_{g}$, say
${\cal G}=\{{G}_{1},{G}_{2},\ldots,{G}_{g-1}\}$, where
$${G}_{i}=\left\{\{a_{i,w},b_{i,w}\}:w\in\left[\frac{g}{2}\right],a_{i,w}<b_{i,w} \right\},$$
and particularly $${G}_{1}=\big\{\{a,a+1\}:a\text{ is odd, } a <g\big\}.$$
%Assume w.l.o.g. that $a_{i,w}<b_{i,w}$ for $i\in[2,g-1]$ and $w\in[\frac{g}{2}]$.
%Assume that$a_{i,w}<b_{i,w}$ for $i \in [g-1]$, $w\in [\frac{g}{2}]$.
Fix $j\in [2,n-1]$ and
apply Strategy B with $t=\frac{g}{2}$ by (\ref{t}).
Since $n\geq g+3$, we have $\frac{n-3}{2} \geq \frac{g}{2}$ and
we choose ${g\over 2}$ distinct sets ${D}_{j+w}$, $w\in[\frac{g}{2}]$, satisfying property (B), where $j+w\in[2,\frac{n-1}{2}]\pmod{\frac{n-3}{2}}$ if $j\in [2,\frac{n-1}{2}]$, and $j+w\in[\frac{n+1}{2},n-1]\pmod{\frac{n-1}{2}}$ otherwise.
%Since $G_{1}=\left\{\{a_{1,w},b_{1,w}\}:w\in\left[\frac{g}{2}\right] \right\}$ is a one-factor of $K_{g}$,
Then we obtain a one-factor of $K_{n\times g}$ without parallel edges
$${B}_{1,j}=\bigcup_{w\in[\frac{g}{2}]}{D}_{j+w}(a_{1,w},b_{1,w}).$$
%where $j+w\in[2,\frac{n-1}{2}]\pmod{\frac{n-3}{2}}$ if $j\in [2,\frac{n-1}{2}]$, and $j+w\in[\frac{n+1}{2},n-1]\pmod{\frac{n-1}{2}}$ otherwise.
%Because the set $\{a_{1,w},b_{1,w}:w\in [\frac{g}{2}]\}$ contains all elements in $[g]$, %Property (B1) is satisfied.  Since $n\geq g+3$, we have $\frac{n-3}{2} \geq \frac{g}{2}$,  Property (B2) is satisfied.
%It follows that each ${B}_{1,j}$ is a one-factor of $K_{n\times g}$ without parallel edges.
When $j$ runs over $[2,n-1]$, we obtain $n-2$  one-factors of $K_{n\times g}$ without parallel edges.

Apply Strategy B with $t=\frac{g}{2}$ again.
%Since $G_{i}$ is a one-factor of $K_{g}$,
For any $i\in [2,g-1]$ and $j\in [n-1]$, define
$${B}_{i,j}=\bigcup_{w\in[\frac{g}{2}]}{D}_{j+w}(a_{i,w},b_{i,w}),$$
where $j+w\in[\frac{n-1}{2}]\pmod{\frac{n-1}{2}}$ if $j\in[\frac{n-1}{2}]$, and $j+w\in[\frac{n+1}{2},n-1]\pmod{\frac{n-1}{2}}$ otherwise.
Note that property (B) is satisfied and $G_{i}$ is a one-factor of $K_{g}$. %the set $\{a_{i,w},b_{i,w}:w\in [\frac{g}{2}]\}$ contains all elements in $[g]$.
So, each  ${B}_{i,j}$ is a one-factor of~$K_{n\times g}$ without parallel edges.

Now, we prove that $$\big\{{A}_{i}:i\in[n]\big\}\bigcup \big\{{B}_{i,j}:i\in[g-1],j\in [n-1],(i,j)\neq (1,1)\big\}$$
forms a one-factorization of $K_{n\times g}$.
By Lemma \ref{d3},
we only need to prove that
any edge $\{(x,a),(y,b)\}$ of $K_{n\times g}$ occurs at least once.\vspace{-0.2cm}
\begin{itemize}
 \item [$(1)$] If $a=b$, there exists an $r\in[n]$ such that $\{x,y\}\in {F}_{r}$ since  ${\cal F}$ is a near one-factorization of $K_{n}$. Then the edge $\{(x,a),(y,a)\}\in {F}_{r}(a,a)\subseteq{A}_{i}$, where $i\in [n]$ and $i+a-1\equiv r\pmod{n}$.\vspace{-0.2cm}
 \item [$(2)$] If $a< b$, there exists an $i\in [g-1]$ such that $\{a,b\}=\{a_{i,w},b_{i,w}\}\in {G}_{i}$ as \ ${\cal G}$ is a one-factorization of $K_{g}$.
     We can find $r\in [n-1]$ such that $(x,y)\in{D}_{r}$ since the sets $D_{1},D_{2},\ldots,D_{n-1}$ form a partition of all ordered pairs of $[n]$ by (\ref{E1}). We divide the discussion as follows.
\begin{itemize}
 \item [$(a)$] If $i=r=1$, then $\{(x,a),(y,b)\}\in {A}_{k}$, where $k\in[n]$ and $k+a-1\equiv x\pmod{n}$.%\vspace{-0.2cm}
 \item [$(b)$] If $i=1$ and $r\in [2,{n-1}]$, then $\{(x,a),(y,b)\}\in {D}_{j+w}(a_{1,w},b_{1,w})\subseteq{B}_{1,j}$, where $j\in[2,\frac{n-1}{2}]$,
        $j+w \equiv r\pmod {\frac{n-3}{2}}$ if  $r\in [2,\frac{n-1}{2}]$ and  $j\in[\frac{n+1}{2},n-1]$,
      $j+w \equiv r\pmod {\frac{n-1}{2}}$ otherwise.%\vspace{-0.2cm}
 \item [$(c)$] If $i\neq1$ and $r\in [{n-1}]$, then $\{(x,a),(y,b)\}\in {D}_{j+w}(a_{i,w},b_{i,w})\subseteq{B}_{i,j}$, where
        $j+w \equiv r\pmod {\frac{n-1}{2}}$ with $j \in[\frac{n-1}{2}]$ if  $r\in [\frac{n-1}{2}]$ and $j\in[\frac{n+1}{2}, n-1]$ otherwise.%\vspace{-0.2cm}
     \end{itemize}\vspace{-0.2cm}
\end{itemize}\vspace{-0.2cm}
In conclusion, an $\mathrm{OF}(n,g)$ is obtained.
\qed

\subsection{Constructions for OF$(2n,g)$}
In this subsection we will construct $\mathrm{OF}(2n,g)$s for relatively large $n$ in three propositions.

\begin{Proposition}\label{OF_g_odd_1}
There exists an $\mathrm{OF}(2n,g)$ for any odd $g\geq 3$ and odd $n\geq g$.
\end{Proposition}
\proof We take a near one-factorization $\{{G}_{1},{G}_{2},\ldots,{G}_{g}\}$ of $K_{g}$, such that in ${G}_{i}$, $i\in[g]$, vertex~$i$ is isolated, where
%List the pairs in $G_{i}$ by
$${G}_{i}=\left\{\{a_{i,w},b_{i,w}\}:w\in\left[\frac{g-1}{2}\right],a_{i,w}<b_{i,w}\right\}.$$
%Assume w.l.o.g. that $a_{i,w}<b_{i,w}$ for $i\in[g]$ and $w\in[\frac{g-1}{2}]$.
Furthermore, let $\{{F}_{1},{F}_{2},\ldots,{F}_{n}\}$ be a near one-factorization of $K_{n}$, such that in ${F}_{r}$, $r\in [n]$, vertex $r$ is isolated.
%We construct $n$ mutually disjoint one-factors of $K_{2n}$ based on ${F}_{r}$, $r\in [n]$.

For any $r\in [n]$, define
$${H}_{r} = \big\{\{x,y\},\{x+n,y+n\}:\{x,y\}\in {F}_{r}\big\}\bigcup\big\{\{r,r+n\}\big\}.$$
Moreover, for any $r\in[n]$, define
$${I}_{r}=\big\{\{x,x+r+n\}:x\in [n]\big\},$$
where $x+r+n\in[n+1,2n]\pmod{n}$. %$x+r+n$ is reduced modulo $n$ to lie in $[n+1,2n]$.
Note that $I_{n}=\big\{\{r,r+n\}:r\in [n]\big\}$. Then it is readily checked that  $\{{H}_{r}:r\in [n]\}\bigcup\{{I}_{r}:r\in[n-1]\}$ forms a one-factorization of $K_{2n}$.
Then our construction proceeds in three steps.

\textbf{Step 1:} For any $i\in [n]$, the set $\{i+a\in[n]\pmod n:a\in[g]\}$ contains $g$ distinct elements in $[n]$ since $n\geq g$.
Note that ${H}_{1},{H}_{2},\ldots,{H}_{n}$ are $n$ pairwise disjoint one-factors of $K_{2n}$.
Applying Strategy A with $s=g$, we have that
$${A}_{i} = \bigcup_{a\in [g]}{H}_{i+a}(a,a), i\in [n],$$
 are $n$ one-factors of $K_{2n\times g}$ without parallel edges,
where $i+a \in [n]\pmod n$. Clearly, every edge of the form $\{(x,a),(y,a)\}$, where $\{x,y\}\in H_{r}$, $r\in[n]$, $a\in[g]$, can be found in the set $\{A_{i}:i\in[n]\}$.

\textbf{Step 2:} For any $j\in [n]$, the set $\{j+a\in[n]\pmod n:a\in [g]\}$ contains $g$ distinct elements in $[n]$ since $n\geq g$. Moreover, ${I}_{1},{I}_{2},\ldots,{I}_{n}$ are $n$ pairwise disjoint one-factors of $K_{2n}$.
We apply Strategy A with $s=g$ again.
Note that (1) $e\cap[n]\neq\emptyset,e\cap[n+1,2n]\neq\emptyset$ for any $e\in I_r,r\in[n]$, and (2) $\{a:a\in [g]\}=\{a+i\in [g] \pmod g:a\in [g]\}=[g]$ for any $i \in [g-1]$.
For any $i \in [g-1]$ and $j\in[n]$, define
$${B}_{i,j}=\bigcup_{a\in [g]}{I}_{j+a}(a,a+i),$$
where $j+a\in[n]\pmod n$ and $a+i\in [g]\pmod g$.
It is immediate that property (A) is satisfied; hence
${B}_{i,j}$, $i \in [g-1]$, $j\in[n]$, are $(g-1)n$ one-factors of $K_{2n\times g}$ without parallel edges. In this step, we run out of all edges of the form
$\{(x,a),(y,b)\}$, where $\{x,y\}\in I_{r}$, $r\in[n]$, $a\neq b\in[g]$.

\textbf{Step 3:}
In order to obtain an $\mathrm{OF}(2n,g)$, we still need $g(n-1)$ one-factors of $K_{2n\times g}$ by Lemma \ref{d3}.
For any $r\in [n-1]$, define ${D}'_{r}=\{(x,y),(x+n,y+n):(x,y)\in {D}_{r}\}$, where~${D}_{r}$ is defined in (\ref{E1}).
Note that $(x,y)\in D'_{r_{1}}$ implies $(y,x)\in D'_{r_{2}}$, where $r_{1}+r_{2}=n$ with $r_{1},r_{2}\in [n-1]$. Then we partition the set $[n-1]$ into two parts $[\frac{n-1}{2}]$ and $[\frac{n+1}{2},n-1]$, from which we define
for any $i \in [g]$ and $j\in [n-1]$,
$${C}_{i,j}=\bigcup_{w \in [\frac{g-1}{2}]}{D}'_{j+w}(a_{i,w},b_{i,w})\bigcup{I}_{j}(i,i),$$
where $j+w\in [\frac{n-1}{2}]\pmod{\frac{n-1}{2}}$ if $ j \in [\frac{n-1}{2}]$ and $j+w\in [\frac{n+1}{2},n-1]\pmod{\frac{n-1}{2}}$  otherwise.
From such a partition $[n-1]=[\frac{n-1}{2}]\bigcup[\frac{n+1}{2},n-1]$, it follows that each ${C}_{i,j}$ dose not contain parallel edges.
Since $\{a_{i,w},b_{i,w}:w \in [\frac{g-1}{2}]\}=[g]\setminus \{i\}$,  each vertex $(x,a)$ with $x\in[2n]$ and $a\in[g]\setminus\{i\}$ is contained in exactly one edge in   $\{{D}'_{j+w}(a_{i,w},b_{i,w}): w \in [\frac{g-1}{2}]\}$;
and the vertex $(x,i)$ with $x\in[2n]$ is contained exactly once in ${I}_{j}(i,i)$.
%In this step, we run out of all edges of the form$\{(x,a),(y,b)\}$, $\{x,y\}\in I_{r}$ for some $r\in[n]$, $a\neq b\in[g]$

Now, we prove that $$\{A_{i}:i\in[n]\}\bigcup\{{B}_{i,j}:i\in[g-1],j\in[n]\}\bigcup\{{C}_{i,j}:i\in[g],j\in[n-1]\}$$ is a one-factorization of $K_{2n\times g}$.
By Lemma \ref{d3},
we only need to prove that
any edge $\{(x,a),(y,b)\}$ with $x<y$ of $K_{2n\times g}$ occurs at least once.
%Let $\{(x,a),(y,b)\}$ be an edge of $K_{2n\times g}$ with $x<y$.
We consider two cases.
\vspace{-0.2cm}
\begin{itemize}
 \item [$(1)$] $x\in [n]$ and $y\in [n+1,2n]$.
 Let $r\in [n]$ and $y-n-x\equiv r\pmod{n}$.
 If $a=b$ and $r=n$, then there exists a $\{x,y\}\in {H}_{x}$, so $\{(x,a),(y,a)\}\in{H}_{x}(a,a)\subseteq{A}_{i}$ with $i\in[n]$ and $i+a\equiv x\pmod{n}$.
  If $a=b$ and $r\neq n$, then there exist $\{x,y\}\in {I}_{r}$ and $\{(x,a),(y,a)\}\in{I}_{r}(a,a)\subseteq{C}_{a, r}$.
     If $a\neq b$, let $i\in[g-1]$ with $b- a\equiv i \pmod {g}$. Then $\{(x,a),(y,b)\}\in{I}_{r}(a,b)\subseteq {B}_{i,j}$, where $j\in[n]$ and $j+a\equiv r\pmod{n}$.\vspace{-0.2cm}
 \item [$(2)$] $\{x,y\}\subseteq [n]$ or $\{x,y\}\subseteq [n+1,2n]$. If $a=b$, there exists an $r\in[n]$ such that $\{x,y\}\in {H}_{r}$. Then $\{(x,a),(y,b)\}\in{H}_{r}(a,a)\subseteq{A}_{i}$ with $i\in[n]$ and $i+a\equiv r\pmod{n}$. If $a\neq b$ and $\{a,b\}=\{a_{i,w},b_{i,w}\}\in G_{i}$, let $r\in[n-1]$ and $x-y\equiv r \pmod {n}$.
     If $r\in [\frac{n-1}{2}]$, then $\{(x,a),(y,b)\}\in {D}_{r}(a_{i,w},b_{i,w})\subseteq{C}_{i,j}$ with $j\in [\frac{n-1}{2}]$ and $j+w\equiv r\pmod{\frac{n-1}{2}}$.
     If $r\in [\frac{n+1}{2},n-1]$, then $\{(x,a),(y,b)\}\in {D}_{r}(a_{i,w},b_{i,w})\subseteq{C}_{i,j}$ with $j\in [\frac{n+1}{2},n-1]$ and $j+w\equiv r\pmod{\frac{n-1}{2}}$.\vspace{-0.2cm}
\end{itemize}
So, the desired $\mathrm{OF}(2n,g)$ is obtained.
\qed

%\newpage
\begin{Proposition}\label{OF_g_even_2}
There exists an $\mathrm{OF}(2n,g)$ for any even $g$ and even  $n\geq2g$.
\end{Proposition}
\proof
 There exists a one-factorization of $K_{g}$, say $\{{G}_{1},{G}_{2},\ldots,{G}_{g-1}\}$. %where ${G}_{i}=\big\{\{a_{i,1},b_{i,1}\},$ $\{a_{i,2},b_{i,2}\},\ldots,\{a_{i,\frac{g}{2}},b_{i,\frac{g}{2}}\}\big\}$ with $a_{i,w}<b_{i,w}$ $(i\in[g-1],w\in[\frac{g}{2}])$.
Let $${G}_{i}=\left\{\{a_{i,w},b_{i,w}\}:w\in\left[\frac{g}{2}\right],a_{i,w}<b_{i,w}\right\}.$$
%Assume w.l.o.g. that $a_{i,w}<b_{i,w}$ for $i\in[g-1]$ and $w\in[\frac{g}{2}]$.
Furthermore, we take a one-factorization of $K_{n}$, say $\{{F}_{1},{F}_{2},\ldots,{F}_{n-1}\}$.
For any $r\in [n-1]$, define a one-factor of $K_{2n}$, say
$${H}_{r} = \big\{\{x,y\},\{x+n,y+n\}:\{x,y\}\in {F}_{r}\big\}.$$
For any $r\in[\frac{n}{2}]$, define
$${I}_{r}=\big\{\{x,x+r\}:x\in[n]\big\},$$
where $x+r\in[\frac{3n}{2}+1,2n]\pmod{\frac{n}{2}}$  if $x\in[\frac{n}{2}]$ and $x+r \in [n+1,\frac{3n}{2}]\pmod{\frac{n}{2}}$ if $x\in[\frac{n}{2}+1, n]$.
For $r\in[\frac{n}{2}+1, n]$, define
$${I}_{r}=\big\{\{x,x+r\}:x\in [n]\big\},$$
where $x+r \in [n+1,\frac{3n}{2}]\pmod{\frac{n}{2}}$ if $x\in[\frac{n}{2}]$ and $x+r\in [\frac{3n}{2}+1,2n]\pmod{\frac{n}{2}}$ if $x\in[\frac{n}{2}+1,n]$.
For $r\in[n,2n-1]$, define ${H}_{r} ={I}_{r-n+1}$.
It can be checked that %$\{\mathcal{H}_{r}:1\leq r \leq n-1\}\cup\{\mathcal{I}_{r}:1\leq r \leq n\}$
 $\{{H}_{r}:r\in [2n-1]\}$ forms a one-factorization of $K_{2n}$.
Then this construction proceeds in three steps.

\textbf{Step 1:} %Note that  ${H}_{r}$, $r\in [2n-1]$ are mutually disjoint one-factors of $K_{2n}$.
Since $g<2n$, for any $i\in [2n-1]$, the set $\{i+a\in [2n-1]\pmod{2n-1}:a\in [g]\}$ contains $g$ distinct elements in $[2n-1]$.
Applying Strategy A with $s=g$ gives that
$${A}_{i} = \bigcup_{a\in[g]}{H}_{i+a}(a,a), i\in[2n-1],$$
are $2n-1$ one-factors of $K_{2n\times g}$ without parallel edges,
where $i+a\in [2n-1]\pmod{2n-1}$.
Now, we run out of all edges of the form
$\{(x,a),(y,a)\}$, $x\neq y \in[2n]$, $a\in[g]$ as $\{{H}_{r}:r\in [2n-1]\}$ is a one-factorization of $K_{2n}$.

\textbf{Step 2:}
For any $r\in [n-1]\setminus\{\frac{n}{2}\}$, define ${D}'_{r}=\{(x,y),(x+n,y+n):(x,y)\in {D}_{r}\}$, where~${D}_{r}$ is defined in (\ref{E1}).
Note that for $r_{1},r_{2}\in [n-1]\setminus\{\frac{n}{2}\}$, $(x,y)\in D'_{r_{1}}$ implies $(y,x)\in D'_{r_{2}}$, where $r_{1}+r_{2}=n$.
For $i\in [g-1]$, $j\in [n-1]\setminus\{\frac{n}{2}\}$, define
$${B}_{i,j}=\bigcup_{w\in[\frac{g}{2}]}{D}'_{j+w}(a_{i,w},b_{i,w}),$$
where $j+w\in [\frac{n}{2}-1]\pmod{\frac{n}{2}-1}$  if $j\in[\frac{n}{2}-1]$ and $j+w\in [\frac{n}{2}+1,n-1]\pmod{\frac{n}{2}-1}$ otherwise.
Since $\frac{g}{2}\le \frac{n}{2}-1$, each ${B}_{i,j}$ does not contain parallel edges.  Let $i,j$ be fixed. For any vertex $(x,a)$ in $K_{2n\times g}$, we can find $w\in [\frac{g}{2}]$ such that $a_{i,w}=a$ or $b_{i,w}=a$ since each ${G}_{i}$ is a one-factor of $K_{g}$.
Then there exists a unique edge in ${D}'_{j+w}(a_{i,w},b_{i,w})\subseteq {B}_{i,j}$ that contains~$(x,a)$. So each~${B}_{i,j}$ is a one-factor of $K_{2n\times g}$ without parallel edges.
In this step, we run out of all edges of the form
$\{(x,a),(y,b)\}$, $\{x,y\}\subseteq [n]$ or $\{x,y\}\in[n+1,2n]$ except $y-x\equiv\frac{n}{2}\pmod{n}$, $a\neq b\in[g]$.
%$(x,y)\in {D}'_{r}$ for some $r\in[n-1]\setminus\{\frac{n}{2}\}$, $a\neq b\in[g]$.

\textbf{Step 3:}
For any $i\in [g-1]$, $j\in [\frac{n}{2}]$, define
$${C}_{i,j}=\bigcup_{a \in [g]}{I}_{j+a}(a,a+i),$$
where the subscript $j+a \in [\frac{n}{2}]\pmod{\frac{n}{2}}$ and the second component $a+i\in[g]\pmod g$. Let~$i,j$ be fixed.
Since $g\leq \frac{n}{2}$, the one-factors ${I}_{j+a},a\in [g]$ used in $C_{i,j}$ are pairwise disjoint. Note that each edge of $I_{r}$, $r\in[\frac{n}{2}]$ intersects both the sets $[n]$ and $[n+1,2n]$.
Further note $\{a:a\in [g]\}=\{a+i\in[g]\pmod g:a\in [g]\}=[g]$. Thus, property (A) is satisfied. Then each~$C_{i,j}$ is a one-factor of $K_{2n\times g}$ without parallel edges by Strategy A.

Denote ${I}_{n+1}=\left\{\{x,x+\frac{n}{2}\}:x \in[\frac{n}{2}]\bigcup[n+1, \frac{3n}{2}]\right\}$.
Note that each edge of ${I}_{r}(a,b)$, $r\in [\frac{n}{2}+1,n+1],a\neq b \in [g]$, has not been used. Next, we distribute all missing edges of  $K_{2n\times g}$.

For any $i\in [g-1]$, define a $g\times 2n$ array over $[g]$ by
$$
M_{i}=\big(m_{a,x}^{(i)}\big)_{g\times 2n}=\left(\underbrace{\begin{matrix}
1 & \cdots & 1  \\
2 & \cdots& 2  \\
\vdots & &\vdots\\
g & \ldots& g \\
\end{matrix}}_\frac{n}{2}
\text{ }
\underbrace{\begin{matrix}
1+i & \cdots & 1+i  \\
2+i & \cdots & 2+i  \\
\vdots & &\vdots \\
g+i & \cdots & g+i \\
\end{matrix}}_n
\text{ }
\underbrace{\begin{matrix}
1+2i & \cdots & 1+2i \\
2+2i & \cdots & 2+2i \\
 \vdots&&\vdots\\
g+2i & \cdots & g+2i \\
\end{matrix}}_\frac{n}{2}
\right), $$
where the addition is taken by modulo $g$.
Note that all columns of $M_{i}$ form permutations of~$1,2,\ldots,g$.

For any $r\in[\frac{n}{2}+1,n+1]$, $a\in[g]$, $i \in [g-1]$, define %construct an almost $1$-regular subgraph of $K_{2n\times g}$ with edge set
$${K}_{a,r}^{i}=\left\{\big\{(x,m_{a,x}^{(i)}),(y,m_{a,y}^{(i)})\big\}:\{x,y\}\in {I}_{r},x<y\right\}.$$
Based on this, for $i \in [g-1]$ and $j\in[\frac{n}{2}+1,n+1]$, define
$${C}_{i,j}=\bigcup_{a\in[g]}{K}_{a,j+a}^{i},$$
where $j+a\in [\frac{n}{2}+1,n+1]\pmod{\frac{n}{2}+1}$.
 It is not difficult to have that Strategy A is applicable and thus each  ${C}_{i,j}$, $i \in [g-1]$, $j\in [\frac{n}{2}+1,n+1]$, is a one-factor of $K_{2n\times g}$ without parallel edges.

Now, we prove that  $$\big\{{A}_{i}:i\in [2n-1]\big\}\bigcup \bigg\{{B}_{i,j}:i\in[g-1],j\in [n-1]\setminus\left\{\frac{n}{2}\right\}\bigg\}\bigcup \big\{{C}_{i,j}:i\in[g-1],j\in[n+1]\big\}$$ forms a one-factorization of $K_{2n\times g}$. %Let $\{(x,a),(y,b)\}$ be an edge in $K_{2n\times g}$.
By Lemma \ref{d3},
we only need to prove that
any edge $\{(x,a),(y,b)\}$ with $x<y$ of $K_{2n\times g}$ occurs at least once.
We have two cases to be considered.
\vspace{-0.2cm}
\begin{itemize}
 \item [$(1)$] $a=b$. We have $\{x,y\}\in {H}_{r}$ for some $r\in[2n-1]$, since $\{{H}_{r}:r\in[2n-1]\}$ is a one-factorization of $K_{2n}$. Then $\{(x,a),(y,b)\}\in{H}_{r}(a,a)\subseteq{A}_{i}$, where $i\in [2n-1]$ and $i+a\equiv r\pmod{2n-1}$.\vspace{-0.2cm}
 \item [$(2)$] $a\neq b$. We consider two possibilities.
 \begin{itemize}
 \item[(i)]
 $(x,y)\in {D}'_{r}$ for some $r\in[n-1]\setminus\{\frac{n}{2}\}$. Let $\{a,b\}=\{a_{i,w},b_{i,w}\}$ for some $i\in[g-1]$ and $w\in [\frac{g}{2}]$.
     If $r\in[\frac{n}{2}-1]$, then $\{(x,a),(y,b)\}\in{D}'_{r}(a_{i,w},b_{i,w})\subseteq{B}_{i,j}$, where $j\in[\frac{n}{2}-1]$ and $j+w\equiv r\pmod{\frac{n}{2}-1}$.
     If $r\in[\frac{n}{2}+1,n-1]$,
     then $\{(x,a),(y,b)\}\in{D}'_{r}(a_{i,w},b_{i,w})\subseteq{B}_{i,j}$, where $j\in[\frac{n}{2}+1,n-1]$ and $j+w\equiv r\pmod{\frac{n}{2}-1}$.
   \item[(ii)]  $\{x,y\}\in {I}_{r}$ for some $r\in [n+1]$. Let $x<y$ and let $b-a \equiv i \pmod{g}$. When $r\in [\frac{n}{2}]$,
     we have $\{(x,a),(y,b)\}\in{I}_{r}(a,a+i)\subseteq {C}_{i,j}$, where $j\in[\frac{n}{2}]$ and $j+a\equiv r\pmod{\frac{n}{2}}$.
     When $r\in[\frac{n}{2}+1, n+1]$, we have $\{(x,a),(y,b)\}\in{K}_{a',r}^{i}\subseteq {C}_{i,j}$, where $j\in[\frac{n}{2}+1,n+1]$, $j+a'\equiv r\pmod{\frac{n}{2}+1}$
      with
     $a'=a$ if $x\in [{n\over 2}]$ and $a'+i\equiv a\pmod g$ otherwise.% if $x\in[{n\over 2}+1,{3n\over 2}]$, and  $a'+2i\equiv a\pmod g$ if $x\in[{3n\over 2}+1,2n]$.
     \vspace{-0.2cm}
\end{itemize}
\end{itemize}
So, we obtain an $\mathrm{OF}(2n,g)$.
\qed

\begin{Proposition}\label{OF_g_odd_2}
There exists an $\mathrm{OF}(2n,g)$ for any odd $g\geq 3$ and even  $n\geq 2g$.
\end{Proposition}
\proof This is analogous to the proof of Proposition \ref{OF_g_even_2}.
Take a near one-factorization $\{{G}_{1},{G}_{2},\ldots,{G}_{g}\}$ of $K_{g}$, such that in ${G}_{i}$, $i\in[g]$, vertex $i$ is isolated. %$${G}_{i}=\big\{\{a_{i,1},b_{i,1}\},\{a_{i,2},b_{i,2}\},\ldots,\{a_{i,\frac{g-1}{2}},b_{i,\frac{g-1}{2}}\}\big\},\text{ } \text{where } a_{i,w}<b_{i,w}\ (i\in[g],w\in\left[\frac{g-1}{2}\right]).$$
Let
$${G}_{i}=\left\{\{a_{i,w},b_{i,w}\}:w\in\left[\frac{g-1}{2}\right],a_{i,w}<b_{i,w} \right\}.$$
%Assume w.l.o.g. that $a_{i,w}<b_{i,w}$ for $i\in[g]$ and $w\in[\frac{g-1}{2}]$.
Moreover, let $\{{F}_{1},{F}_{2},\ldots,{F}_{n-1}\}$ be a one-factorization of $K_{n}$. As in the proof of Proposition \ref{OF_g_even_2}, we obtain a one-factorization of $K_{2n}$, namely, $\{{H}_{r}:r \in [n-1]\}\bigcup\{{I}_{r}:r\in [n]\}$.
Then proceed in three steps.

\textbf{Step 1:} Let ${H}_{n}={I}_{\frac{n}{2}}$ and ${H}_{n+1}={I}_{n}$. Applying Strategy A with $s=g$ gives that
$${A}_{i} = \bigcup_{a \in[g]}{H}_{i+a}(a,a),i\in[n+1],$$
are $n+1$ one-factors of $K_{2n\times g}$ without parallel edges,
where $i+a \in [n+1]\pmod{n+1}$.
Now, we run out of all edges of the form
$\{(x,a),(y,a)\}$, where $\{x,y\}\in H_{r}$, $r\in[n+1]$, $a\in[g]$.

\textbf{Step 2:}
For any $r\in [n-1]\setminus\{\frac{n}{2}\}$, define ${D}'_{r}=\{(x,y),(x+n,y+n):(x,y)\in {D}_{r}\}$, where~${D}_{r}$ is defined in (\ref{E1}).
%Note that $(x,y)\in D'_{r_{1}}$ implies $(y,x)\in D'_{r_{2}}$, where $r_{1}+r_{2}=n$ with $r_{1},r_{2}\in [\frac{n}{2}-1] \bigcup[\frac{n}{2}+1,n-1]$. Then we partition the set $[\frac{n}{2}-1]\bigcup[\frac{n}{2}+1,n-1]$ into two parts $[\frac{n}{2}-1]$ and $[\frac{n}{2}+1,n-1]$.
For any $i\in [g]$, $j\in [n-1]\setminus\{\frac{n}{2}\}$, define

$${B}_{i,j}=\bigcup_{w \in [\frac{g-1}{2}]}{D}'_{j+w}(a_{i,w},b_{i,w})\bigcup{I}_{j}(i,i),$$
where $j+w\in [\frac{n}{2}-1]\pmod{\frac{n}{2}-1}$ if $j\in[\frac{n}{2}-1]$, and $j+w\in [\frac{n}{2}+1,n-1]\pmod{\frac{n}{2}-1}$ otherwise.
It can be checked that $\{{B}_{i,j}:i\in [g],j\in [n-1]\setminus\{\frac{n}{2}\}\}$ contains $g(n-2)$ one-factors of $K_{2n\times g}$ without parallel edges.
In this step, we run out of all edges of the form
$\{(x,a),(y,a)\}$ with $\{x,y\}\in I_{r}$, $r\in[n-1]\setminus\{\frac{n}{2}\}$, $a\in[g]$ and of the form $\{(x,a),(y,b)\}$ with $\{x,y\}\subseteq [n]$ or $\{x,y\}\in[n+1,2n]$ except $y-x\equiv\frac{n}{2}\pmod{n}$, $a\neq b\in[g]$.

\textbf{Step 3:} This step is completely the same with Step 3 in the proof of Proposition \ref{OF_g_even_2}. So we obtain $(g-1)(n+1)$ one-factors of $K_{2n\times g}$
 without parallel edges, namely, ${C}_{i,j},i\in[g-1],j\in[n+1]$,
%For any $i\in [g-1]$, $j\in [\frac{n}{2}]$, define
%$${C}_{i,j}=\bigcup_{a \in [g]}{I}_{j+a}(a,a+i), \text{ where } j+a \in [\frac{n}{2}]\pmod{\frac{n}{2}}.$$
%Denote ${I}_{n+1}=\left\{\{x,x+\frac{n}{2}\}:x \in[\frac{n}{2}]\bigcup[n+1, n+\frac{n}{2}]\right\}$.
%For any $r\in[\frac{n}{2}+1,n+1]$, $a\in[g]$, $i \in [g-1]$, let
%$${K}_{a,r}^{i}=\{\{(x,m_{a,x}^{(i)}),(y,m_{a,y}^{(i)})\}:\{x,y\}\in {I}_{r}\},$$ where $M_i=\big(m_{a,x}^{(i)}\big)_{g\times 2n}$ is defined in Step 3 of the proof of Proposition \ref{OF_g_even_2}.
%Then, for $i \in [g-1]$ and $j\in[\frac{n}{2}+1,n+1]$, define
%$${C}_{i,j}=\bigcup_{a\in[g]}{K}_{a,j+a}^{i},$$
%where $j+a\in[\frac{n}{2}+1,n+1]\pmod{\frac{n}{2}+1}$.
%Then  each $C_{i,j}$ is a one-factor of $K_{2n\times g}$ without parallel edges by Strategy A.

Finally, it is readily checked that  $$\big\{{A}_{i}:i\in [n+1]\big\}\bigcup \bigg\{{B}_{i,j}:i\in[g],j\in [n-1]\setminus\left\{\frac{n}{2}\right\}\bigg\}\bigcup \big\{{C}_{i,j}:i\in[g-1],j\in[n+1]\big\}$$
forms  an $\mathrm{OF}(2n,g)$. (See the final part in the proof of  Proposition \ref{OF_g_even_2} for similar details.)\qed

\subsection{Main theorem}

The constructions in Subsections 4.1-4.5 give rise to one-factorizations of $K_{n\times g}$ with distance three  for all even $gn$ with $n>g$, only leaving a small gap of $n$ to be resolved.
We state this main result in detail in a theorem.

\begin{Theorem}\label{results_OF} For one-factorizations of $K_{n\times g}$ with distance $d=3$, we have the following:
\begin{itemize}
 \item [$(1)$] For $g=2,3$, there exists an $\mathrm{OF}(n,g)$ if and only if $n> g$ and $gn$ is even.
\item [$(2)$] Let $g\ge 4$ be even. There exists an $\mathrm{OF}(n,g)$ if and only if $n> g$, possibly except odd $n$ with $g+3\leq n\leq 2g-3$ and even $n$ with
 $g+2\leq n\leq 4g-4$.
\item [$(3)$]
  Let $g\geq 5$ be odd. There exists an $\mathrm{OF}(n,g)$ if and only if $n>g$ is even, possibly except that $g+1\leq n \leq 2g-4$ if $n\equiv 2\pmod 4$  and that $g+1\leq n\leq 4g-4$ if  $n\equiv 0\pmod 4$.
\item [$(4)$]
If $g$ is an odd prime power, then an $\mathrm{OF}(n,g)$ exists if and only if $n>g$ is even, possibly except that $g+3\leq n \leq 2g-4$ if $n\equiv 2\pmod 4$ and that $g+3\leq n\leq 4g-4$ if $n\equiv 0\pmod 4$.
  %In particular, if $g$ is an odd prime power, then an $\mathrm{OF}(n,g)$ exists if and only if $n>g$ is even, possibly except that $g+3\leq n \leq 2g-4$ if $n\equiv 2\pmod 4$ and that $g+3\leq n\leq 4g-4$ if $n\equiv 0\pmod 4$.
\end{itemize}
\end{Theorem}

\proof By Lemma \ref{d3}, an $\mathrm{OF}(n,g)$ exists only if $ng$ is even and $n> g$. %Then we only need to show the sufficiency in the following proof.

By Proposition \ref{OF_g_even_1}, there exists an $\mathrm{OF}(n,2)$ for any odd $n$ with $n\ge 5$.
Combine with the $\mathrm{OF}(n,2)$ for $n=3,4$ in  Example \ref{OF_3_n_2} and Lemma \ref{OF_n_odd} to have that (1) holds for $g=2$.
An $\mathrm{OF}(n,3)$ exists for any $n\equiv 2\pmod 4$ with $n\geq 6$ by Proposition \ref{OF_g_odd_1}; an $\mathrm{OF}(n,3)$ exists for any  $n\equiv 0\pmod 4$ with $n\geq 12$ by Proposition~\ref{OF_g_odd_2}. Lemma \ref{OF_n_even} gives the existence of an $\mathrm{OF}(4,3)$. Thus we also have an $\mathrm{OF}(8,3)$ by Lemma \ref{OF_double}. It then follows that (1) also holds for $g=3$.

Let $g\geq 4$ be even. If $n$ is odd, then by Proposition \ref{OF_g_even_1}, there exists an $\mathrm{OF}(n,g)$ for any odd $n$ with $n\geq 2g-1$. Moreover, by Lemma \ref{OF_n_odd}, there exists an $\mathrm{OF}(g+1,g)$.
So there exists an $\mathrm{OF}(n,g)$ for any even $g\ge 4$ and odd $n>g$, possibly except for $g+3\leq n\leq 2g-3$. If $n$ is even, then by Proposition \ref{OF_g_even_2}, there exists an $\mathrm{OF}(n,g)$ for any  $n\geq4g$. By Proposition \ref{OF_g_even_1}, there exists an $\mathrm{OF}(2g-1,g)$. Applying the doubling construction in Lemma \ref{OF_double}, we obtain an $\mathrm{OF}(4g-2,g)$.
This concludes the proof of (2).
%To sum up, we arrive at the conclusion in part (2).

%Let $g$ be even. Firstly, let $n$ be odd. Then by Proposition \ref{OF_g_even_1}, there exists an $\mathrm{OF}(n,g)$ for any odd $n$ with $n\geq 2g-1$ if $g\geq 4$ or $n\ge 5$ if $g=2$. Moreover, there exists an $\mathrm{OF}(g+1,g)$ by Lemma \ref{OF_n_odd}. So there exists an $\mathrm{OF}(n,g)$ for any even $g\ge 4$ and odd $n>g$, possibly except those $n$ with $g+3\leq n\leq 2g-3$. Secondly,  let $n\equiv 2\pmod 4$. Apply the doubling construction in Lemma \ref{OF_double} with the known $\mathrm{OF}(n/2,g)$. Then we obtain an $\mathrm{OF}(n,g)$ for any $n\geq 4g-2$ if $g\ge 4$ and $n\ge 10$ if  $g=2$. Thirdly, let $n\equiv 0\pmod 4$. Then there exists an $\mathrm{OF}(n,g)$ for any  $n\geq4g$ by Proposition \ref{OF_g_even_2}. To sup up, we arrive at the conclusion in part (2).

If $g\geq 5$ is odd, then by Proposition \ref{OF_g_odd_1}, there exists an $\mathrm{OF}(n,g)$ for any $n\equiv 2\pmod 4$ with $n\geq 2g$.
By Proposition~\ref{OF_g_odd_2}, there exists an $\mathrm{OF}(n,g)$ for any  $n\equiv 0\pmod 4$ with $n\geq 4g$.
So there exists an $\mathrm{OF}(n,g)$ if and only if $n>g$ is even, possibly except when $g+1\leq n \leq 2g-4$ if $n\equiv 2\pmod 4$ and when $g+1\leq n\leq 4g-4$ if $n\equiv 0\pmod 4$. This completes the proof of (3).

If $g$ is an odd prime power, then by Lemma \ref{OF_n_even}, there exists an $\mathrm{OF}(g+1,g)$. Together with (3),  we complete the proof of (4).
\qed

%The following lists an $\mathrm{LGS}(1,2,4,3)$, each part is a $\mathrm{GS}(1,2,4,3)$ with the alphabet $\mathbb{Z}_{4}$.
%\begin{center}
%\begin{tabular}{l l l l l l l l }
%\{(1, 1), (2, 1)\}& \{(1, 2), (3, 1)\}& \{(1, 3), (4, 1)\}& \{(2, 2), (4, 2)\}& \{(2, 3), (3, 2)\}& \{(3, 3), (4, 3)\}\\
%\{(1, 1), (2, 2)\}& \{(1, 3), (3, 1)\}& \{(1, 2), (4, 1)\}& \{(2, 1), (4, 2)\}& \{(3, 2), (4, 3)\}& \{(2, 3), (3, 3)\}\\
%\{(1, 1), (2, 3)\}& \{(2, 1), (3, 1)\}& \{(2, 2), (4, 1)\}& \{(3, 2), (4, 2)\}& \{(1, 2), (3, 3)\}& \{(1, 3), (4, 3)\}\\
%\{(1, 1), (3, 1)\}& \{(2, 1), (3, 3)\}& \{(3, 2), (4, 1)\}& \{(1, 3), (4, 2)\}& \{(1, 2), (2, 2)\}& \{(2, 3), (4, 3)\}\\
%\{(1, 1), (3, 3)\}& \{(2, 1), (3, 2)\}& \{(3, 1), (4, 1)\}& \{(1, 2), (4, 2)\}& \{(2, 2), (4, 3)\}& \{(1, 3), (2, 3)\}\\
%\{(1, 1), (4, 3)\}& \{(1, 2), (2, 1)\}& \{(2, 2), (3, 1)\}& \{(1, 3), (3, 2)\}& \{(3, 3), (4, 1)\}& \{(2, 3), (4, 2)\}\\
%\{(1, 1), (4, 1)\}& \{(2, 1), (4, 3)\}& \{(3, 1), (4, 2)\}& \{(2, 2), (3, 2)\}& \{(1, 2), (2, 3)\}& \{(1, 3), (3, 3)\}\\
%\{(1, 1), (4, 2)\}& \{(1, 3), (2, 1)\}& \{(3, 1), (4, 3)\}& \{(1, 2), (3, 2)\}& \{(2, 2), (3, 3)\}& \{(2, 3), (4, 1)\}\\
%\{(1, 1), (3, 2)\}& \{(2, 1), (4, 1)\}& \{(3, 3), (4, 2)\}& \{(1, 2), (4, 3)\}& \{(1, 3), (2, 2)\}& \{(2, 3), (3, 1)\}\\
%\end{tabular}
%\end{center}
%\qed

\section{Concluding remarks}

Considering multipartite graphs through coding-theoretic lens,
we study one-factorizations of complete multipartite graphs subject to distance constraints in this paper. More generally, we investigate the problem of optimal decompositions of  $K_{n\times g}$ into the largest subgraphs with a fixed minimum distance. With respect to distances two and four, optimal decompositions  are explored with complete solutions provided in Theorems \ref{d2} and \ref{d4}. With respect to distance three,  Theorems \ref{ODAR_even_n} and~\ref{ODAR_odd_n} show that the complete multipartite graph $K_{n\times g}$ can be decomposed into~$g^2$ copies of the largest subgraphs for all $n\le g$. Theorem \ref{results_OF} shows that there are one-factorizations of $K_{n\times g}$  with distance three for all even $gn$ with $n>g$, only leaving a small gap of $n$ (in terms of~$g$) to be resolved.

%\begin{table}[h!]
%\setlength{\abovecaptionskip}{0cm}
%\setlength{\belowcaptionskip}{+0.2cm} % µ÷Õû±êÌâÏ·½µÄ¼ä¾à
% \begin{center}
 %  \caption{Summary of unresolved OF$(n,g)$ }
 %  \begin{tabular}{c|c|c|c|c}
  % \toprule[2pt]
    % \textbf{g} & \textbf{$n\equiv 0\pmod{4}$} & \textbf{$n\equiv 2\pmod{4}$}& \textbf{$n\equiv 1\pmod{4}$}& \textbf{$n\equiv 3\pmod{4}$}\\
     %\hline
     %4& \multicolumn{4}{c}{6,8,12}\\
     %\hline
     %5&\multicolumn{4}{c}{8,16}\\
     %\hline
     %6&\multicolumn{4}{c}{8,9,10,12,16,18,20}\\
     %\hline
     %7&\multicolumn{4}{c}{10,12,20,24}\\
     %\hline
     %$g\geq 6$ is even&\multicolumn{2}{c|}{$g+2\leq n\leq 4g-4$}&\multicolumn{2}{c}{$g+3\leq n\leq 2g-3$}\\
     %\hline
     %$g\geq 7$ is odd&\multicolumn{4}{c}{10,12,20,24}\\
     %\toprule[2pt]
   %\end{tabular}
 %\end{center}
%\end{table}

To fill in the gaps that we leave in Theorem \ref{results_OF}, for a fixed $g$,  $\mathrm{OF}(n,g)$s with small primes~$n$ may be constructed directly. Then it is promising to get all missing $\mathrm{OF}(n,g)$s by applying the doubling construction and product construction in Lemmas \ref{OF_double} and \ref{OF-mul-construction}, especially for relatively small $g$. However, systematic approaches are expected to be developed to handle all remaining cases. We observed  a large literature concerning factorizations of various hypergraphs since the seminal work of Baranyai \cite{Baranyai}, see for instance \cite{Bahmanian2, Bahmanian,  Amalgamations}. It is worthwhile to consider how this sort of approach works again if a distance constraint is imposed.
 To extend the work of this paper, we put forward the problem of one-factorizations of complete multipartite  hypergraphs with distance constraints and the problem of decomposing multipartite  hypergraphs into the largest sub-hypergraphs with a fixed minimum distance.
%Our method owes as much to design theory as it does to graph theory and coding theory.

%\clearpage

%\appendix
%\section{An  in Lemma \ref{LGS12n2}}
%{\noindent }

\clearpage

\end{document}